\documentclass[12pt, reqno]{amsart}
\usepackage{amssymb}
\usepackage{graphicx}
\usepackage{amssymb}
 \usepackage{amsthm}
\usepackage{epsf}
\usepackage{amsbsy,amsmath}
\usepackage{mathtools}
\usepackage{mathrsfs}
\usepackage{amsfonts}
\usepackage{amssymb}
\usepackage{enumitem}
\usepackage{eucal}
\usepackage{graphics,mathrsfs}
\usepackage{amsthm}
\usepackage{secdot}
\usepackage{esint}
\usepackage{hyperref}
\usepackage{varwidth}
\usepackage{tasks}

\theoremstyle{plain}
\newtheorem{theorem}{Theorem}[section]

\newtheorem{lemma}[theorem]{Lemma}
\newtheorem{proposition}[theorem]{Proposition}

\theoremstyle{definition}
\newtheorem{definition}[theorem]{Definition}%\ignorespaces}

\newtheorem{remark}[theorem]{Remark}

\allowdisplaybreaks

\usepackage{xcolor}
\renewcommand{\d}{\:\! \mathrm{d}}

\long\def\symbolfootnote[#1]#2{\begingroup
\def\thefootnote{\fnsymbol{footnote}}\footnote[#1]{#2}\endgroup}

\numberwithin{equation}{section}
\begin{document}
\title[A subelliptic nonlocal equations on stratified Lie groups] { Compact Embeddings, Eigenvalue Problems, and  subelliptic Brezis-Nirenberg equations involving singularity on stratified Lie groups}
\author{Sekhar Ghosh, Vishvesh Kumar and Michael Ruzhansky}
\address[ Sekhar Ghosh]{Department of Mathematics: Analysis, Logic and Discrete Mathematics, Ghent University, Ghent, Belgium\newline and \newline Statistics and Mathematics Unit, Indian Statistical Institute Bangalore, Bengaluru, 560059, India}
\email{sekharghosh1234@gmail.com / sekhar.ghosh@ugent.be}
\address[Vishvesh Kumar]{Department of Mathematics: Analysis, Logic and Discrete Mathematics, Ghent University, Ghent, Belgium}
\email{vishveshmishra@gmail.com / vishvesh.kumar@ugent.be}

\address[Michael Ruzhansky]{Department of Mathematics: Analysis, Logic and Discrete Mathematics, Ghent University, Ghent, Belgium\newline and \newline
School of Mathematical Sciences, Queen Mary University of London, United Kingdom}
\email{michael.ruzhansky@ugent.be}
\thanks{{\em 2020 Mathematics Subject Classification: } 35R03, 35H20, 35P30, 22E30, 35R11, 35J75}

\keywords{Stratified Lie groups; Heisenberg group; Fractional $p$-sub-Laplacian; Brezis-Nirenberg equations, Sobolev-Rellich-Kondrachov type embeddings; Eigenvalue problems; Nehari manifold; Singularity.}

\maketitle

\begin{abstract} 
 The purpose of this paper is twofold: first  we study an eigenvalue problem for the fractional $p$-sub-Laplacian over the  fractional Folland-Stein-Sobolev spaces on stratified Lie groups. We apply variational methods to investigate  the eigenvalue problems. We conclude the positivity of the first eigenfunction via the strong minimum principle for the fractional $p$-sub-Laplacian. Moreover, we deduce that the first eigenvalue is simple and isolated. Secondly, utilising established properties, we prove the existence of at least two weak solutions via the Nehari manifold technique to  a class of  subelliptic singular problems associated with the fractional  $p$-sub-Laplacian on stratified Lie groups. We also investigate the boundedness of  positive weak solutions to the considered problem via the Moser iteration technique. The results obtained here are also new even for the case $p=2$ with $\mathbb{G}$ being the Heisenberg group.
\end{abstract}

\tableofcontents
	\section{Introduction}

	The study of nonlocal elliptic partial differential equations (PDEs) and developments of the corresponding tools have been well explored in the Euclidean setting  during the last few decades. Apart from the mathematical point of view, the theory of PDEs associated with  nonlocal (or fractional) operators witnessed vast applications in different fields of applied sciences. We list a few (in fact a tiny fraction of them) of such applications involving fractional models like the L\'{e}vy processes in probability theory, in finance, image processing, in anomalous equations, porous medium equations, Cahn-Hilliard equations and Allen-Cahn equations, etc. Interested readers may refer to \cite{ASS16, AMRT10, LPGSGZMCMAK20, TC03} and the references therein. These models have been one of the primary context  to study nonlocal PDEs both theoretically and numerically.

	One of the most important tools to study PDEs over bounded domains is the embeddings of Sobolev spaces into Lebesgue spaces. It says, ``If $\Omega\subset\mathbb{R}^N$ is open, then for $0<s<1<p<\infty$ with $N>ps$, the fractional Sobolev space $W^{s,p}(\Omega)$ is continuously embedded into $L^{q}(\Omega)$ for all $q\in[1,Np/(N-ps)]$. In addition, if $\Omega$ is bounded and is an extension domain, then the embedding is compact for all $q\in[1,Np/(N-ps))$." The compact embedding plays a crucial role for obtaining the existence of solutions of some PDEs. We refer the readers to see \cite{NPV12} for a well presented study of the fractional Sobolev spaces and the properties of the fractional $p$-Laplacian and its applications to PDEs. One can also consult  \cite{B11, E10} for the theory and tools developed for the classical Sobolev spaces.
	
	The Sobolev spaces (also known as Folland-Stein spaces) on stratified Lie groups were  first considered by Folland \cite{F75} and then several further properties have been obtained in the book by Folland and Stein \cite{FS82}.  The reader may  refer to several monographs devoted to the study of such spaces and the corresponding subelliptic operators \cite{BLU07, FR16, RS19}. For Sobolev embeddings  of Folland-Stein spaces over bounded domains of stratified Lie groups, we refer to \cite{CDG93}. Recently, the fractional Sobolev type inequality and the corresponding Sobolev embeddings were investigated in \cite{AM18} for weighted fractional Sobolev spaces on the Heisenberg group $\mathbb{H}^N$, whereas in \cite{KD20}, the authors established the fractional Sobolev type inequalities on stratified Lie groups (or homogeneous Carnot groups). In \cite {AM18}, the authors established the compact embeddings of Sobolev spaces $W_0^{s,p,\alpha}(\Omega)$ into Lebesgue spaces $L^p(\Omega)$ over a bounded extension domain $\Omega\subset\mathbb{H}^N$. We recall here the definition of an extension domain: A domain $\Omega \subset {\mathbb{G}}$ is said to be an extension domain of $W^{s,p}_0(\Omega)$ (see Section \ref{s2} for the definition) if for every $f \in W^{s,p}_0 (\Omega) $ there exist a $\tilde{f} \in W^{s,p}_0({\mathbb{G}})$ such that $\tilde{f}|_\Omega=f$ and $\|\tilde{f}\|_{W^{s,p}_0({\mathbb{G}})}\leq C_{Q,s,p, \Omega} \|f\|_{W^{s,p}_0(\Omega)},$ where $C_{Q,s,p, \Omega}$ is a positive constant depending only on $Q, s, p, \Omega.$   The extension property of a domain plays a crucial role in establishing such compact embeddings of the Sobolev spaces into  Lebesgue spaces ({\it cf.} Theorem 2.4, Lemma 5.1 in \cite{NPV12}). Recently, Zhou \cite{Z15} studied the characterizations of $(s,p)$-extension domains and embedding domains for the fractional Sobolev space on $\mathbb{R}^N$. To the best of our knowledge, we do not have such characterization for an arbitrary bounded domain in the case of stratified Lie groups.  In fact, because of the existence of characteristic points, the problem of finding classes of extension domains in stratified Lie groups is highly non-trival and there are essentially no examples for step $3$ and higher (see \cite{CG98}) . Thus, to overcome this issue, we will work with the fractional Sobolev space $X_0^{s,p}(\Omega)$ with vanishing trace (See Section \ref{s2} for the definition). 
	
 We first state the following embedding result for the fractional Sobolev space $X_0^{s,p}(\Omega).$
\begin{theorem}\label{l-3}
Let $\mathbb{G}$ be a stratified Lie group of homogeneous dimension $Q.$ Let $0<s<1<p<\infty$ and $Q>sp.$ Let $\Omega\subset\mathbb{G}$ be an open subset. Then the fractional Sobolev space $X_0^{s, p}(\Omega)$ is continuously embedded in $L^r(\Omega)$ for $p\leq r\leq p_s^*$, where $p_s^*:=\frac{Qp}{Q-sp}$, that is, there exists a positive constant $C=C(Q,s,p, \Omega)$ such that for all $u\in X_0^{s, p}(\Omega)$, we have
$$\|u\|_{L^r(\Omega)}\leq C \|u\|_{X_0^{s,p}(\Omega)}.$$
Moreover, if $\Omega$ is bounded, then the embedding
\begin{align}
    X_0^{s,p}(\Omega) \hookrightarrow L^r(\Omega)
\end{align}
is continuous for all $r\in[1,p_s^*]$ and is compact for all $r\in[1,p_s^*)$.
\end{theorem}

 It was pointed out to us by the referee of this paper that there may be a relation between Theorem \ref{l-3} and the results in the recent paper \cite{BK22} combined with \cite{KYZ11} dealing the fractional Sobolev spaces defined on metric measure spaces satisfying various conditions (typically, but not always, a Poincar\'e inequality and doubling condition), see also \cite{Pio00} and \cite{GS22}. One such example of a metric measure space is a stratified Lie group. However, it is not completely clear how the result in \cite{BK22} applies to our spaces $X_0^{s,p}$ since the definition of this space is different. Therefore, for the benefit of readers, we include  a simple and direct proof of embedding theorems in Appendix A (Section \ref{appA}) which makes use of group structures such as the group translation and regularisation process via group convolution and dilations for this particular setting of stratified Lie groups. We follow the ideas of \cite{NPV12} to establish the continuous embedding whereas the compact embedding will be proved based on the idea originated by \cite{GL92}. We also refer \cite{AYY22, AWYY23} for embedding results on function spaces defined on spaces of homogeneous type.

	In this paper, we now  aim to apply Theorem \ref{l-3} to study the nonlinear Dirichlet eigenvalue problem on stratified Lie groups. The earliest known study of Dirichlet eigenvalue problems involving the $p$-Laplacian on $\mathbb{R}^N$ is due to Lindqvist \cite{L90}, where the author investigated the simplicity and isolatedness of the first eigenvalue of the following problem:
	\begin{align}\label{p-lind}		\Delta_{p} u&+\nu|u|^{p-2} u=0,~\text{in}~\Omega,\nonumber\\
		u&=0~\text{ on }~\partial\Omega.
	\end{align}

	Lindqvist further showed that the first eigenfunction of the problem \eqref{p-lind} is strictly positive on any arbitrary bounded domain $\Omega$. This study is directly related to the corresponding Rayleigh quotient of the energy given by the following expression:
	\begin{equation}
		\mathcal{R}(u):=\cfrac{\int_{\Omega}{|\nabla u|^p}dx}{\int_{\Omega}|u(x)|^p dx},~u\in C_c^{\infty}(\Omega).
	\end{equation}
	The nonlocal counterpart of the above problem \eqref{p-lind} was explored by Lindgren and Lindqvist  \cite{LL14}, and Franzina and Palatucci \cite{FP14}. After that, this topic received  an extensive attention. For instance, we cite  \cite{L05, AA87, BP16, LL14, FP14} just to mention a few of names toward the development of the eigenvalue problem.
	
	As per our knowledge, the study of eigenvalue problems for the subelliptic setting is very limited in the literature. The earliest traces of such studies are due to \cite{MS78, FP81}. Thereafter, there has been some progress in this direction involving the $p$-sub-Laplacian on the Heisenberg group, for instance, see \cite{HL08, WPH09}. Recently, there is an elevation of interest in the study of eigenvalue problems involving subelliptic operators on stratified Lie groups. We refer to \cite {CC21, HL08, FL10} and the references therein. 
	
	In this paper, we study the following nonlinear nonlocal Dirichlet eigenvalue problem involving the fractional $p$-sub-Laplacian on stratified Lie groups,
\begin{align} \label{pro1.3}
		(-\Delta_{p,{\mathbb{G}}})^s u&=\nu|u|^{p-2} u,~\text{in}~\Omega,\nonumber\\
		u&=0~\text{ in }~{\mathbb{G}}\setminus\Omega.
	\end{align}

In this direction, we first establish the existence of a minimizer for the Rayleigh quotient, namely, the existence of the first eigenfunction. Then,  similar to the classical case, we prove some important  properties of the first eigenfunction and the first eigenvalue of the problem \eqref{pro1.3}, which are listed below in the form of the following result.
\begin{theorem}\label{ev-mainthm}
	Let $0<s<1<p<\infty$ and let $\Omega\subset{\mathbb{G}}$ be a bounded domain of a stratified Lie group $\mathbb{G}$ of homogeneous dimension $Q$. Then for $Q>sp$, we have the following properties.
	\begin{enumerate}[label=(\roman*)]
		\item The first eigenfunction of the problem \eqref{pro1.3} is strictly positive.
		\item The first eigenvalue $\lambda_1$ of the problem \eqref{pro1.3} is simple and the corresponding eigenfunction $\phi_1$ is the only eigenfunction of constant sign, that is, if $u$ is an eigenfunction associated to an eigenvalue $\nu>\lambda_1(\Omega)$, then $u$ must be sign-changing.
		
		\item The first eigenvalue $\lambda_1$ of the problem \eqref{pro1.3} is isolated.
	\end{enumerate}	
\end{theorem}

Among the key ingredients to prove Theorem \ref{ev-mainthm} are a strong minimum principle (Theorem \ref{min p}) and  logarithmic estimates (Lemma \ref{k-log-lemma}).

Now, as a combined application of Theorem \ref{l-3} and Theorem \ref{ev-mainthm}, we turn our attention to the following problem  involving the fractional $p$-sub-Laplacian on the stratified Lie group ${\mathbb{G}}$:
\begin{align}\label{problem}
	\left(-\Delta_{p, {\mathbb{G}}}\right)^s u&=\frac{\lambda f(x)}{u^{\delta}}+g(x)u^{q}  \text { in } \Omega, \nonumber\\
	u&>0  \text { in } \Omega,\\
	u&=0  \text { in } {\mathbb{G}}\setminus\Omega, \nonumber
\end{align}
where $\Omega$ is a bounded  domain in ${\mathbb{G}}$, $\lambda>0$, $1<p<Q$, $0<s, \delta<1<p<q+1<p_s^{*}.$ Here $Q$ denotes the homogeneous dimension of ${\mathbb{G}},$ $p_s^{*}:=\frac{Q p}{Q-sp}$ denotes the critical Sobolev exponent, and $\left(-\Delta_{p, {\mathbb{G}}}\right)^s$ is the fractional $p$-sub-Laplacian ({\it ref.} Section \ref{s2}). The weight functions $f,g\in L^{\infty}(\Omega)$ are strictly positive.

 The problem of the type \eqref{problem} is usually referred to as the Brezis-Nirenberg type problem \cite{BN83} in the literature. Before we briefly recall some studies done in the Euclidean case, let us first discuss the motivation to consider Brezis-Nirenberg type problem on stratified Lie groups setting.  The primary motivation to investigate the Brezis-Nirenberg problem in the classical Euclidean setting (i.e., $p=2$ and $s=1$) comes from the fact that it resembles variation problems in differential geometry and physics. One such celebrated example is the  {\it Yamabe problem} on a Riemannian manifolds. There are many other examples that are directly related to the Brezis-Nirenberg problem; for example, existences of extremal functions for functional inequalities and existence of non-minimal solutions for Yang-Mills functions and $H$-system (see \cite{BN83}). The pioneering investigation of {\it CR Yamabe problem} was started by Jerison and Lee in their seminal work \cite{JL}. It is well-known that the Heisenberg group (simplest example of a stratified Lie group) plays the same role in the CR geometry as the Euclidean space in conformal geometry. Naturally, the analysis on stratified Lie groups proved to be a fundamental tool in the resolution of the CR Yamabe problem. Therefore, a great deal of interest has been shown to studying subelliptic PDEs on stratified Lie groups. Recently, several researchers have considered the    fractional CR Yamabe problem and problems around it; see \cite{GQ13,CW17,FMMT15, KMW17, K20, CHY21} and references therein. These aforementioned developments naturally encourage one for studying the Brezis-Nirenberg type problem \eqref{problem} on stratified Lie groups. Apart from this, it is also worth mentioning that the investigation of  problems of type \eqref{problem} is closely related to the existence of best constant in functional inequalities, e.g. see \cite{GU21} and references therein.

 On the other hand, it was noted in the celebrated paper \cite{RS76} by Rothschild and Stein that nilpotent Lie groups play an important role in deriving sharp subelliptic estimates for differential operators on manifolds. In view of the Rothschild-Stein lifting theorem, a general H\"ormander's sums of squares of vector fields on manifolds can be approximated by a  sup-Laplacian on some stratified Lie group (see also, \cite{F77} and \cite{Roth83}). This makes
it crucial to study  partial differential equations on stratified Lie groups and led to several interesting and promising works amalgamating the Lie group theory with the  analysis on partial differential equations. Moreover,  in recent decades, there is a rapidly growing interest for sub-Laplacians on stratified Lie groups because
these operators appear not only in theoretical settings (see e.g. Gromov \cite{Gro} or Danielli,
Garofalo and Nhieu \cite{DGN07} for general expositions from different points of view), but also in application settings such as mathematical models of crystal material and human vision (see,  \cite{C98} and \cite{CGS04}).

It is almost impossible to enlist all such studies dealing with existence, multiplicity and regularity of solutions but we will mention some of the pivotal studies that motivated us to  consider this problem \eqref{problem} in the subelliptic setting on stratified Lie groups. These studies  are primarily divided into two cases, namely, $\lambda=0$ and $\lambda>0$. For $\lambda>0$, $g=0$, $p=2$, $s=1$ and $\delta>0$, i.e., the purely singular problem involving Laplacian  was initially tossed up by Crandall et al. \cite{CRT77} following a pseudo-plastic model in the bounded domain $\Omega\subset\mathbb{R}^N$ with the Dirichlet boundary condition. Moving forward with the same setting, Lazer and McKenna \cite{LM91} observed that one can expect a $W_0^{1,2}(\Omega)$ solution if and only if $0<\delta<3.$ Later, in \cite{YD13} it was proved that the exponent $\delta=3$ proposed in \cite{LM91} is optimal to obtain a $W_0^{1,2}(\Omega)$ solution. The nonlocal counterparts of these type of PDEs were studied in Canino et al. \cite{CMSS17} for all $p\in(1,\infty)$ and $s\in(0,1)$. For further references on the study of purely singular problems we refer to \cite{BO09, CMSS17} and the references therein. It is noteworthy to mention that the problem \eqref{problem} with $g=0$ always possess a unique solution for $\lambda,\delta>0$. The study of the multiplicity and regularity of solutions was widely considered for $\lambda\geq0,$ see \cite{AF17,BS15, DP97,HSS08,MS16} and the references therein.

We now emphasize  the study of the existence of solutions to PDEs associated with subelliptic operators.  In \cite{ALT20}, the authors established the existence of solutions to the following problem involving the sub-Laplacian on the Heisenberg group:
\begin{align}
	-\Delta_{\mathbb{H}^N} u&=\frac{\lambda f(x)}{u^{\delta}} \text { in } \Omega, \nonumber\\
	u&>0  \text { in } \Omega,\\
	u&=0  \text { in } \partial\Omega, \nonumber
\end{align}
They applied the fixed point theorem argument and a weak convergence method to deduce existence of solutions. The nonlinear extension of the aforementioned problem, that is, the singular problem with the $p$-sub-Laplacian was investigated in \cite{GU21} in the setting of stratified Lie groups for $0<\delta<1$. In both of these studies, the authors used the weak convergence method to establish the existence of solution.
In \cite{WW18}, the author considered subelliptic problem associated with sub-Laplacian on the Heisenberg group with mixed singular and power type nonlinearity. They established the existence  of solution using the moving plane method. The authors \cite{YY12} extended this to the Carnot groups. In \cite{H20}, Han studied existence and nonexistence results for positive solutions to an integral type Brezis-Nirenberg type problem on the Heisenberg groups. Ruzhansky et al. \cite{RST22} established the global existence and blow-up of the positive solutions to a nonlinear porous medium equation over stratified Lie groups. In \cite{BPV22} the authors characterised the existence of unique positive weak solution  for subelliptic Dirichlet problems. A few more results dealing with the existence and multiplicity of solutions over the Heisenberg groups and stratified Lie groups  can be found in \cite{P19, PT22, BFP20a, BR17, FBR17, L19, L07,PT21,L15,FF15} and references listed therein.  Finally, we cite \cite{CKKSS22, EGKSS22} and references therein for the study of non-homogeneous fractional $p$-Laplacian on metric measure spaces. The study of existence and multiplicity of weak solutions mainly uses the variational tools, such as mountain pass theorem, Nehari manifold techniques, etc. 

In this study we employ the Nehari manifold method \cite{ST10, HSS04} to establish the multiplicity of solutions to the problem \eqref{problem}. The result is stated as follows.

\begin{theorem}\label{main-thm}
	Let $\Omega$ be a bounded  domain of a stratified Lie group ${\mathbb{G}}$ of homogeneous dimension $Q$, and let $0<s, \delta<1<p<q+1<p_s^{*}:=\frac{Qp}{Q-ps}$, $Q>ps$. Then there exists $\Lambda>0$ such that for all $\lambda \in\left(0, \Lambda\right)$ the problem \eqref{problem} admits at least two non-negative solutions in $X_0^{s,p}(\Omega)$.
\end{theorem}

 Let us make a few more comments on results of this paper before concluding the introduction. In this paper, our main focus is to study subelliptic eigenvalue problem and  the Brezis-Nirenberg type problem on stratified Lie groups. But, we emphasise that  the proofs and statements of Theorem \ref{ev-mainthm} and Theorem \ref{main-thm} can easily be adopted with suitable modifications in the case of metric measure space, at least, in  case when the metric measure space is doubling and satisfies a Poincar\'e inequality. 

The paper is organized as follows: In the next section we present basics of the analysis on stratified Lie groups along with function spaces defined on them.  In Section \ref{sec4}, we study the fractional $p$-sub-Laplacian eigenvalue problem on stratified Lie groups. The existence of (at least) two solutions of the nonlocal singular problem by using Nehari manifold technique is analysed in Section \ref{sec5}. The last section consists of showing the boundedness of solution by employing the Moser iteration followed by a comparison principle.  
\section{Preliminaries: Stratified Lie groups and  Sobolev spaces}\label{s2}

This section is devoted to recapitulating 
some basic notations and concepts related to stratified Lie groups and the fractional Sobolev spaces defined on them.  There are many ways to introduce the notion of stratified Lie groups, for instance one may refer to  books and monographs \cite{FS82,BLU07,FR16,RS19}. In his seminal paper \cite{F75}, Folland extensively investigated the  properties of function spaces on these groups. The monographs \cite{FR16} deals with the theory associated to higher order invariant operators, namely, the Rockland operators on graded Lie groups. For precise studies and properties on stratified Lie group, we refer \cite{GL92, F75, FS82, BLU07, FR16}.

\begin{definition}\label{d1}
	A Lie group $\mathbb{G}$ (on $\mathbb{R}^N$) is said to be homogeneous if, for each $\lambda>0$, there exists an automorphism $D_{\lambda}:\mathbb{G} \rightarrow\mathbb{G}$  defined by $D_{\lambda}(x)=(\lambda^{r_1}x_1, \lambda^{r_2}x_2,..., \lambda^{r_N}x_N)$ for $r_i>0,\,\forall\, i=1,2,...,N$. The map $D_\lambda$ is called a {\it dilation} on $\mathbb{G}$. 
\end{definition}
For simplicity, we sometimes prefer to use the notation $\lambda x$ to denote the dilation $D_{\lambda}x$. Note that, if $\lambda x$ is a dilation then $\lambda^r x$ is also a dilation. The number $Q=r_1+r_2+...+r_N$ is called the homogeneous dimension of the homogeneous  Lie group $\mathbb{G}$ and the natural number $N$ represents the topological dimension of $\mathbb{G}.$ The Haar measure on $\mathbb{G}$ is denoted by $dx$ and it is nothing but the usual Lebesgue measure on $\mathbb{R}^N.$

\begin{definition}\label{d-carnot}
	A homogeneous Lie group $\mathbb{G} = (\mathbb{R}^N, \circ)$ is called a stratified Lie group (or a homogeneous Carnot group) if the following two conditions are fulfilled:
	\begin{enumerate}[label=(\roman*)]
		\item For some natural numbers $N_1 + N_2+ ... + N_k = N$ the decomposition $\mathbb{R}^N = \mathbb{R}^{N_1} \times \mathbb{R}^{N_2} \times ... \times \mathbb{R}^{N_k}$ holds, and for each $\lambda > 0$ there exists a dilation of the form $D_{\lambda}(x)=  (\lambda x^{(1)}, \lambda^2 x^{(2)},..., \lambda^{k}x^{(k)})$ which is an automorphism of the group $\mathbb{G}$. Here $x^{(i)}\in \mathbb{R}^{N_i}$ for $i = 1,2, ..., k$.
		
		\item With $N_1$ the same as in the above decomposition of $\mathbb{R}^N$, let $X_1, ..., X_{N_1}$ be the left invariant vector fields on $\mathbb{G}$ such that $X_k(0) = \frac{\partial}{\partial x_k}|_0$ for $k = 1, ..., N_1$. Then the H\"{o}rmander condition $rank(Lie\{X_1, ..., X_{N_1} \}) = N$
		holds for every $x \in \mathbb{R}^{N}$. In other words, the Lie algebra corresponding to the Lie group $\mathbb{G}$ is spanned by the iterated commutators of $X_1, ..., X_{N_1}$.
	\end{enumerate}
\end{definition}

Here $k$ is called the step of the homogeneous Carnot group.  Note that, in the case of stratified Lie groups, the homogeneous dimension becomes $Q=\sum_{i=1}^{i=k}iN_i$. Furthermore, the left-invariant vector fields $X_j$ satisfy the divergence theorem and they can be written explicitly as 
\begin{equation}\label{d-vec fld}
	X_i=\frac{\partial}{\partial x_i^{(1)}} + \sum_{j=2}^{k}\sum_{l=1}^{N_1} a^{(j)}_{i,l}(x^1, x^2, ..., x^{j-1})\frac{\partial}{\partial x_l^{(j)}}.
\end{equation}

For simplicity, we set $n=N_1$ in  the above Definition \ref{d-carnot}. 

An absolutely continuous curve $\gamma:[0,1]\rightarrow\mathbb{R}$ is said to be admissible, if there exist functions $c_i:[0,1]:\rightarrow\mathbb{R}$, for $i=1,2...,n$ such that
$${\dot{\gamma}(t)}=\sum_{i=1}^{i=n}c_i(t)X_i(\gamma(t))~\text{and}~ \sum_{i=1}^{i=n}c_i(t)^2\leq1.$$
Observe that the functions $c_i$ may not be unique as the vector fields $X_i$ may not be linearly independent. For any $x,y\in\mathbb{G}$ the Carnot-Carath\'{e}odory distance is defined as
$$\rho_{cc}(x,y)=\inf\{l>0: ~\text{there exists an admissible}~ \gamma:[0,l]\rightarrow\mathbb{G} ~\text{with}~ \gamma(0)=x ~\text{\&}~ \gamma(l)=y \}.$$
We define $\rho_{cc}(x,y)=0$, if no such curve exists. This $\rho_{cc}$ is not a metric in general but the H\"{o}rmander condition for the vector fields $X_1,X_2,...X_{N_1}$ ensures that $\rho_{cc}$ is a metric. The space $(\mathbb{G}, \rho_{cc})$ is is known as a Carnot-Carath\'{e}odory space.

%The distance of an element $\xi\in\mathbb{G}$ is given by $\rho=dist(\xi,0)=\left((|x|^2+|y|^2)^2+t^2\right)^{\frac{1}{4}}$. This distance also induces the norm (known as Koranyi norm) on the Heisenberg group $\mathbb{G}$.

Let us now define the quasi-norm on the homogeneous Carnot group $\mathbb{G}$.
\begin{definition}\label{d-quasi-norm}
	A continuous function $|\cdot|: \mathbb{G} \rightarrow \mathbb{R}^{+}$ is said to be a homogeneous quasi-norm on a homogeneous Lie group $\mathbb{G}$ if it satisfies the following conditions:
	\begin{enumerate}[label=(\roman*)]
		
		\item (definiteness): $|x| = 0$ if and only if $x = 0$.
		\item (symmetric): $|x^{-1}| = |x|$ for all $x \in\mathbb{G}$, and 
		\item ($1$-homogeneous): $|\lambda x| = \lambda |x|$ for all $x \in\mathbb{G}$ and $\lambda>0$.
		
	\end{enumerate}
\end{definition}

An example of a quasi-norm on $\mathbb{G}$ is the norm defined as $d(x):=\rho_{cc}(x, 0),\,\, x\in \mathbb{G}$, where $\rho$ is a Carnot-Carath\'{e}odory distance related to H\"ormander vector fields on $\mathbb{G}.$ It is known that all homogeneous quasi-norms are equivalent on $\mathbb{G}.$ In this paper we will work with a left-invariant homogeneous distance $d(x, y):=|y^{-1} \circ x|$ for all $x, y \in \mathbb{G}$ induced by the homogeneous quasi-norm of $\mathbb{G}.$

The sub-Laplacian (or Horizontal Laplacian)  on $\mathbb{G}$ is defined as
\begin{equation}\label{d-sub-lap}
	\mathcal{L}:=X_{1}^{2}+\cdots+X_{N_1}^{2}.
\end{equation}
The horizontal gradient on $\mathbb{G}$ is defined as
\begin{equation}\label{d-h-grad}
	\nabla_{\mathbb{G}}:=\left(X_{1}, X_{2}, \cdots, X_{N_1}\right).
\end{equation}
The horizontal divergence on $\mathbb{G}$ is defined by
\begin{equation}\label{d-h-div}
	\operatorname{div}_{\mathbb{G}} v:=\nabla_{\mathbb{G}} \cdot v.
\end{equation}
For $p \in(1,+\infty)$, we define the $p$-sub-Laplacian on the stratified Lie group $\mathbb{G}$ as
\begin{equation}\label{d-p-sub}
	\Delta_{\mathbb{G},p} u:=\operatorname{div}_{\mathbb{G}}\left(\left|\nabla_{\mathbb{G}} u\right|^{p-2} \nabla_{\mathbb{G}} u\right).
\end{equation}

Let $\Omega$ be a Haar measurable subset of $\mathbb{G}$. Then $\mu(D_{\lambda}(\Omega))=\lambda^{Q}\mu(\Omega)$ where $\mu(\Omega)$ is the Haar measure of $\Omega$. The quasi-ball of radius $r$ centered at $x\in\mathbb{G}$ with respect to the quasi-norm $|\cdot|$ is defined as
\begin{equation}\label{d-ball}
	B(x, r)=\left\{y \in \mathbb{G}: \left|y^{-1} \circ x\right|<r\right\}.
\end{equation}
Observe that $B(x, r)$ can be obtained by the left-translation by $x$ of the ball $B(0, r)$. Furthermore, $B(0, r)$ is the image under the dilation $D_{r}$ of $B(0,1)$. Thus, we have $\mu(B(x,r))=r^Q$ for all $x\in \mathbb{G}$.

%Consider the unit quasi-sphere defined by
%\begin{equation}\label{d-sphere}
%	\mathfrak{S}(0):=\{x \in \mathbb{G}: |x|=1\}.
%\end{equation}

%Then there is a unique positive Borel measure $\sigma$ on the quasi-sphere $\mathfrak{S}(0)$ such that for all $f \in L^{1}(\mathbb{G})$, we have
%\begin{equation}\label{d-polar}
%	\int_{\mathbb{G}} f(x) d x=\int_{0}^{\infty} \int_{\mathfrak{S}(0)} f(r y) r^{Q-1} d\sigma(y) dr.
%\end{equation}

We are now in a position to define the notion of fractional Sobolev-Folland-Stein type spaces related to our study.

Let $\Omega \subset {\mathbb{G}}$ be an open subset. Then for $0<s<1< p<\infty$, the fractional Sobolev space $W^{s,p}(\Omega)$ on stratified groups is defined as 

\begin{equation}
	W^{s,p}(\Omega)=\{u\in L^{p}(\Omega): [u]_{s, p,\Omega}<\infty\},
\end{equation}
endowed with the norm 
\begin{equation}
	\|u\|_{W^{s,p}(\Omega)}=\|u\|_{L^p(\Omega)}+[u]_{s,p,\Omega},
\end{equation}
where $[u]_{s, p,\Omega}$ denotes the Gagliardo semi-norm defined by
\begin{equation}
	[u]_{s, p,\Omega}:=\left(\int_{\Omega} \int_{\Omega} \frac{|u(x)-u(y)|^{p}}{\left|y^{-1} x\right|^{Q+ps}} dxdy\right)^{\frac{1}{p}}<\infty.
\end{equation}

Observe that for all $\phi \in C_c^{\infty}(\Omega)$, we have $[u]_{s, p,\Omega}<\infty$. We define the space  $W_0^{s,p}(\Omega)$ as the closure of $C_c^{\infty}(\Omega)$ with respect to the  norm $\|u\|_{W^{s,p}(\Omega)}$. We would like to point out that $W_0^{s,p}(\mathbb{G})=W^{s,p}(\mathbb{G})$.

Now for an open bounded subset $\Omega \subset {\mathbb{G}}$, define the space $X_0^{s,p}(\Omega)$  as the closure of $C_c^{\infty}(\Omega)$ with respect to the  norm $\|u\|_{L^p(\Omega)}+[u]_{s, p,\mathbb{G}}$. Note that the spaces $X_0^{s,p}(\Omega)$ and $W_0^{s,p}(\Omega)$ are different even in the Euclidean case unless $\Omega$ is an extension domain (see \cite{NPV12}). 
\begin{lemma}\label{l-2}
The space $X_0^{s,p}(\Omega)$ is a  reflexive Banach space for $1< p<\infty$.
\end{lemma}

The space $X_0^{s,p}(\Omega)$ can be equivalently defined with respect to the homogeneous norm $[u]_{s, p,\mathbb{G}}$. Indeed, for $u\in C_c^{\infty}(\Omega)$ and $B_r\subset {\mathbb{G}}\setminus\Omega$, we have 
\begin{equation}
		\left|u(x)\right|^{p}= {|y^{-1}x|^{Q+ps}}  \frac{\left|u(x)-u(y)\right|^{p}}{|y^{-1}x|^{Q+ps}}
	\end{equation}
	for all $x\in \Omega$ and $y\in B_r$. Integrating  first with respect to $y$ and then with respect to $x$, we obtain,
	\begin{equation}
		\int_{\Omega}\left|u(x)\right|^{p}dx\leq \frac{diam(\Omega\cup B_r)^{Q+ps}}{|B_r|} \int_\Omega\int_{B_r} \frac{\left|u(x)-u(y)\right|^{p}}{|y^{-1}x|^{Q+ps}}dxdy.
	\end{equation}
	Now define $$C=C(Q,s,p,\Omega)=\inf\Big\{\frac{diam(\Omega\cup B)^{Q+ps}}{|B|}: B\subset {\mathbb{G}}\setminus\Omega~\text{is a ball}\Big\}.$$
	Then we have the following Poincar\'{e} type inequality,
	\begin{equation}\label{poin}
		\|u\|_{L^p(\Omega)}^p\leq C[u]_{s, p,\mathbb{G}}^p.
	\end{equation}
	
	This confirms that the space $X_0^{s,p}(\Omega)$ can be defined as a closure of $C_c^{\infty}(\Omega)$ with respect to the homogeneous norm $[u]_{s,p, \mathbb{G}}$. That is $$\|u\|_{X_0^{s,p}(\Omega)}\cong[u]_{s, p,\mathbb{G}}~\text{for all}~u\in X_0^{s,p}(\Omega).$$
	
	Moreover, the construction of the space $X_0^{s,p}(\Omega)$ suggests that by assigning $u=0$ in ${\mathbb{G}}\setminus\Omega$ for $u\in X_0^{s,p}(\Omega)$, we conclude that the inclusion map $i:X_0^{s,p}(\Omega)\rightarrow {W}^{s,p}(\mathbb{G})$ is well-defined and continuous.

Note that, in general $X_0^{s,p}(\Omega)\subset\{u\in W^{s, p}({\mathbb{G}}): u=0~\text{in}~{\mathbb{G}}\setminus\Omega \}$. From the equivalence of the norms and being the closure of $C_c^{\infty}(\Omega)$ with respect to the norm $\|\cdot\|_{L^p(\Omega)}+ [u]_{s, p,\mathbb{G}}$, we can represent $X_0^{s,p}(\Omega)$ as follows:  $$X_0^{s,p}(\Omega)=\{u\in W^{s, p}({\mathbb{G}}): u=0~\text{in}~{\mathbb{G}}\setminus\Omega \},$$ 
whenever $\Omega$ is bounded with at least continuous boundary $\partial\Omega$. For $u\in X_0^{s,p}(\Omega)$,
\begin{align*}
    [u]_{s, p,\mathbb{G}}&=\iint_{\mathbb{G} \times \mathbb{G}} \frac{|u(x)-u(y)|^{p}}{\left|y^{-1} x\right|^{Q+p s}} dxdy=\iint_{\mathbb{G} \times \mathbb{G}\setminus (\Omega^c\times\Omega^c)} \frac{|u(x)-u(y)|^{p}}{\left|y^{-1} x\right|^{Q+p s}} dxdy.
\end{align*}

%Henceforth, for the sake of simplicity, we will use the following notations unless otherwise mentioned: $$X_0:=X_0^{s,p}(\Omega);~\|u\|:=\|u\|_{X_0^{s,p}(\Omega)};~[u]_{s, p}:=[u]_{s, p,\mathbb{G}}~\text{ and}~\|u\|_{L^p(\Omega)}:=\|u\|_p.$$

%Let us now define the meaning of $u=0$ on $\partial\Omega$ in $X_0$.
%\begin{definition}
%	Assume $u = 0$ in ${\mathbb{G}}\setminus\Omega$. Then, we define $u\leq 0$  on $\partial\Omega$, if $(u-\epsilon)_+\in X_0$ for all $\epsilon>0$. We say $u=0$ on $\partial\Omega$ if $u\geq0$ as well as $u\leq0$ on $\partial\Omega$.
%\end{definition}
We conclude this section with the following two definitions. For $s\in(0,1)$ and $p\in(1,\infty)$, we define the fractional $p$-sub-Laplacian  as 
\begin{equation}
	\left(-\Delta_{p,{\mathbb{G}}}\right)^s u(x):=C_{Q,s, p}\,\,  P.V. \int_{{\mathbb{G}} } \frac{|u(x)-u(y)|^{p-2}(u(x)-u(y))}{\left|y^{-1} x\right|^{Q+p s}} dy, \quad x \in {\mathbb{G}}.
\end{equation}

For any $\varphi\in X_0^{s,p}(\Omega)$, we have

\begin{equation}
	\langle\left(-\Delta_{p,{\mathbb{G}}}\right)^s u,\varphi\rangle= \iint_{\mathbb{G} \times \mathbb{G}} \frac{|u(x)-u(y)|^{p-2}(u(x)-u(y))(\varphi(x)-\varphi(y))}{\left|y^{-1} x\right|^{Q+p s}} dxdy.
\end{equation}

The simplest example of a stratified Lie group is the Heisenberg group $\mathbb{H}^N$ with the underlying manifold $\mathbb{R}^{2N+1}:=\mathbb{R}^N\times\mathbb{R}^N\times\mathbb{R}$ for $N\in \mathbb{N}$. For $(x, y, t),(x', y', t')\in \mathbb{H}^N$ the multiplication in $\mathbb{H}^N$ is given by

\begin{equation*}
	(x, y, t) \circ(x', y', t')=(x+x', y+y', t+t'+2 (\langle x',y\rangle)-\langle x,y'\rangle),
\end{equation*}
where $(x, y, t), (x', y', t') \in \mathbb{R}^{N}\times \mathbb{R}^{N}\times \mathbb{R}$ and $\langle\cdot,\cdot\rangle$ represents the inner product on $\mathbb{R}^N$. The homogeneous structure of  the Heisenberg group $\mathbb{H}^N$ is provided by the following dilation, for $\lambda>0,$
\begin{equation*}
	D_{\lambda}(x,y,t)=(\lambda x, \lambda y, \lambda^2 t).
\end{equation*}
the  homogeneous dimension $Q$ of $\mathbb{H}^N$ is given by $2N+2:= N+N+2$ while the topological dimension of $\mathbb{H}^N$ is $2N+1.$
The left-invariant vector fields $\{X_i,Y_i\}_{i=1}^N$ defined below form a basis for the Lie algebra corresponding to the Heisenberg group $\mathbb{H}^N$:
\begin{align}
	X_i&=\frac{\partial}{\partial x_i}+2y_i\frac{\partial}{\partial t};
	Y_i=\frac{\partial}{\partial y_i}-2x_i\frac{\partial}{\partial t}~\text{and}~
	T=\frac{\partial}{\partial t}, ~\text{for}~i=1, 2,..., N.
\end{align}
It is easy to see that $[X_i,Y_i]=-4T$ for $i=1,2,...,N$ and $$[X_i,X_j]=[Y_i,Y_j]=[X_i,Y_j]=[X_i,T]=[Y_j,T]=0$$ for all $i\neq j$ and these vector fields satisfy the H\"{o}rmander rank condition. Consequently, the sub-Laplacian on $\mathbb{H}^N$ is given by 
$$\mathfrak{L}_{\mathbb{H}^N}:=\sum_{i=1}^N (X_i^2+Y_i^2).$$

\section{Fractional $p$-sub-Laplacian eigenvalue problem on stratified Lie groups} \label{sec4}

%The best Sobolev constant is defined as
%\begin{equation}\label{s-1}
%	S_s=\inf \left\{\|u\|^{p}: u \in X_0,|u|_{p_{s}^{*}}^{p}=1\right\}.
%\end{equation}
This section is devoted to the study of eigenvalue problems associated to the fractional $p$-sub-Laplacian on stratified Lie groups. 
Let us consider the following PDE on a stratified Lie group ${\mathbb{G}}$:
\begin{align}\label{l-eigen}
	\left(-\Delta_{p,{\mathbb{G}}}\right)^{s} u&=\nu|u|^{p-2} u,~\text{in}~\Omega,\nonumber\\
	u&=0~\text{ in }~{\mathbb{G}}\setminus\Omega,
\end{align}
where $\nu\in\mathbb{R}$ and $\Omega$ is bounded domain in ${\mathbb{G}}$. The problem \eqref{l-eigen} is usually referred to as the fractional $p$-sub-Laplacian (or $(s,p)$-sub-Laplacian) eigenvalue problem. 
\begin{definition}
We say that $u\in X_0^{s,p}(\Omega)$ is a weak solution to \eqref{l-eigen} if, for each $\phi\in C_c^{\infty}(\Omega),$ we have

\begin{equation}\label{ev-soln}
	\langle\left(-\Delta_{p,{\mathbb{G}}}\right)^s u,\phi\rangle=\nu\int_{\Omega} |u|^{p-2}u\phi dx.
\end{equation}
A nontrivial solution to \eqref{ev-soln}  is known as the $(s,p)$-sub-Laplacian eigenfunctions corresponding to an $(s,p)$-sub-Laplacian eigenvalue $\nu$.
\end{definition}
Such eigenfunctions are directly related to the following minimization problem of the Rayleigh quotient $\mathcal{R}$ defined by
\begin{equation}\label{ev-raleigh}
\mathcal{R}(u):=\cfrac{\iint_{\mathbb{G}\times\mathbb{G}}\frac{|u(x)-u(y)|^p}{|y^{-1}x|^{Q+ps}}dxdy}{\int_{\Omega}|u(x)|^p dx},~u\in C_c^{\infty}(\Omega).
\end{equation}

Observe that a minimizer for the Rayleigh quotient does not change its sign. This follows immediately from the triangle inequality \begin{equation*}
		|u(x)-u(y)|>||u(x)|-|u(y)||~\text{whenever}~ u(x)u(y)<0.
	\end{equation*}
Consider the space $\mathcal{S}$ defined as
\begin{equation}
\mathcal{S}=\{u\in X_0^{s,p}(\Omega): \|u\|_p=1\}.
\end{equation}
Then the eigenfunctions of \eqref{l-eigen} are the minimizers of the following energy functional on $\mathcal{S}$:
\begin{equation}\label{ev-energy}
	I(u)=\iint_{\mathbb{G}\times \mathbb{G}}\frac{|u(x)-u(y)|^p}{|y^{-1}x|^{Q+ps}}dxdy.	
\end{equation}
In particular, the eigenfunctions of the problem \eqref{l-eigen} coincides with the critical points of $I$ on the space $\mathcal{S}$.

We define the first eigenvalue or the least eigenvalue  $\lambda_1(\Omega)$ over $\Omega$ as
\begin{align}\label{ev-1}
\lambda_1(\Omega)=&\inf\{\mathcal{R}(\phi): \phi \in C_c^{\infty}(\Omega)\}\\
&\text{or}\nonumber\\
\lambda_1(\Omega)=&\inf\{I(u): u \in \mathcal{S}\}.
\end{align}
 
 Recall the Sobolev inequality \eqref{embed-cont} which is given by
 $$\left(\int_{\Omega}|u(x)|^{p^*} dx \right)^{\frac{1}{p_s^*}} \leq C(Q, p, s) \left( \int_{\mathbb{G}} \int_{ \mathbb{G}} \frac{|u(x)-u(y)|^p}{|y^{-1}x|^{Q+ps}} dx\,dy \right)^{\frac{1}{p}}.$$
 From the H\"older inequality with the exponent $\frac{p^*_s}{p}$ and $\frac{p_s^*}{p_s^*-p}$  we obtain the following inequality which assures that the first eigenvalue $\lambda_1(\Omega)$ of \eqref{l-eigen} is positive:
\begin{align}\label{dadaineq}
    C(Q, p, s)^{-p}\,|\Omega|^{-\frac{ps}{Q}} \int_{\Omega} |u(x)|^p\,dx \leq \int_{\mathbb{G}} \int_{ \mathbb{G}} \frac{|u(x)-u(y)|^p}{|y^{-1}x|^{Q+ps}} dx\,dy.
\end{align}
Thus, by definition all eigenvalues of \eqref{l-eigen} are positive.
 The weak solution of  \eqref{l-eigen} corresponding to $\nu=\lambda_1$ is called the first  eigenfunction of \eqref{l-eigen}.

 We now state the following existence result for the problem \eqref{l-eigen}.

\begin{theorem}\label{ev-thm1}
	Let $0<s<1<p<\infty$ and let $\Omega$ be a bounded domain of a stratified Lie group $\mathbb{G}$ of homogeneous dimension $Q$. Then for $Q>ps$, there exists a nonnegative minimizer $\phi_1$ of \eqref{ev-energy} in $X_0^{s,p}(\Omega)$ and $\phi_1$ is a weak solution to the problem \eqref{l-eigen} for $\nu=\lambda_1(\Omega)$. Moreover, $\phi_1\in L^{\infty}(\Omega)$. Furthermore, there exists $C=C(Q, p,s)$ such that $\lambda_1(\Omega) \geq C |\Omega|^{-\frac{ps}{Q}}.$
\end{theorem}
\begin{proof}	The proof for existence is straightforward from the direct method of the calculus of variations. Suppose $\{u_n\}$ is a minimizing sequence for $I$. Then, by the Sobolev inequality, we have that $\{u_n\}$ is bounded in $X_0^{s,p}(\Omega)$. Thanks to the reflexivity of $X_0^{s,p}(\Omega)$, we get $\phi_1\in X_0^{s,p}(\Omega)$ such that up to a subsequence $u_n\rightharpoonup \phi_1$ weakly in $X_0^{s,p}(\Omega)$ and therefore, $u_n\rightarrow \phi_1$ strongly in $(X_0^{s,p}(\Omega))^{'} :=X_0^{-s,p'}(\Omega)$. Thus, Theorem \ref{l-3} implies that  $u_n\rightarrow \phi_1$ strongly in $L^p(\Omega)$ and $u_n\rightarrow \phi_1$ a.e. in $\Omega$ and $u_n\rightarrow \phi_1$ strongly in $L^{p'}(\Omega)$, where $p'=\frac{p}{p-1}.$  To prove the strong convergence, we show that $\|u_n\|_{X_0^{s,p}(\Omega)}\rightarrow\|\phi_1\|_{X_0^{s,p}(\Omega)}$. The weak convergence implies that
	\begin{equation} \label{4.7eq}
	\langle\left(-\Delta_{p,{\mathbb{G}}}\right)^s u_n- \left(-\Delta_{p,{\mathbb{G}}}\right)^s \phi_1, u_n-\phi_1\rangle\rightarrow0.
	\end{equation}
We will use  the following inequality from \eqref{append}:
	\begin{equation} \label{4.8eq}
		\langle\left(-\Delta_{p,{\mathbb{G}}}\right)^s u_1- \left(-\Delta_{p,{\mathbb{G}}}\right)^s u_2, u_1-u_2\rangle\geq C_p\begin{cases}
			[u_1-u_2]_{s,p}^p,&\text{if}~p\geq2\\
			\frac{[u_1-u_2]_{s,p}^2}{\left([u_1]_{s,p}^p+[u_2]_{s,p}^p\right)^{\frac{2-p}{p}}},&\text{if}~1<p<2.
		\end{cases}
	\end{equation}
Thus, by combining these two inequalities \eqref{4.7eq} and \eqref{4.8eq}, we obtain $\|u_n\|_{X_0^{s,p}(\Omega)}\rightarrow\|\phi_1\|_{X_0^{s,p}(\Omega)}$ and therefore, by using the uniform convexity, we conclude $u_n\rightarrow \phi_1$ strongly in $X_0^{s,p}(\Omega)$. In addition to this, we also observe that $I(|\phi_1|)=I(\phi_1)$. Thus we conclude that the solutions are nonnegative. Indeed, we have
	\begin{align*}
	    \lambda_1(\Omega)&=\inf_{u\in \mathcal{S}}\int_{\mathbb{G}} \int_{ \mathbb{G}} \frac{|u(x)-u(y)|^p}{|y^{-1}x|^{Q+ps}}dxdy\\
	    &\leq \int_{\mathbb{G}} \int_{ \mathbb{G}} \frac{||\phi_1(x)|-|\phi_1(y)||^p}{|y^{-1}x|^{Q+ps}}dxdy\\
	    &\leq \int_{\mathbb{G}} \int_{ \mathbb{G}} \frac{|\phi_1(x)-\phi_1(y)|^p}{|y^{-1}x|^{Q+ps}}\\
	    &=\lambda_1(\Omega).
	\end{align*}
	Thus, $|\phi_1|$ is also minimizes $I$  over $\mathcal{S}$. Therefore, we may conclude that the first eigenfunction of \eqref{l-eigen} can be chosen to be non-negative.

By taking $\lambda=0,$ $g(x)=\nu$ and $q=p$ in the problem \eqref{problem} and from the Lemma \ref{bounded} (see Section \ref{sec6}), we deduce that every solutions of the eigenvalue problem \eqref{ev-1} are uniformly bounded.
\end{proof}

\begin{theorem}\label{sign-changing}
	Let $0<s<1<p<\infty.$ Assume that $\Omega\subset{\mathbb{G}}$ is a bounded domain of a stratified Lie group $\mathbb{G}$. Let $v\in X_0^{s,p}(\Omega)$ solve \eqref{l-eigen} and assume that $v>0,$ and let $\nu$ be the corresponding eigenvalue of $v$. Then we have
	\begin{equation}
	\nu=\lambda_1(\Omega),
	\end{equation}
	where $\lambda_1(\Omega)=\inf \{I(\phi): \phi \in X_0^{s,p}(\Omega)\}$. In particular, any eigenfunction corresponding to an eigenvalue $\nu>\lambda_1(\Omega)$ must be sign-changing.
\end{theorem}
\begin{proof}
	 For every nonnegative $u,v\in X_0^{s,p}(\Omega)$, we claim that 
	\begin{equation}\label{ev-z}
	I(z(t))\leq(1-t)I(v)+tI(u),~\forall\, t\in[0,1],
	\end{equation}
	where $z(t)= \left((1-t)v^p(x) + tu^p(x) \right)^{1/p}, ~\forall\,t\in[0,1].$ Let us first establish the above   inequality.
The estimate follows immediately by considering the $\ell_p$-norm of $z(t)$ over $\mathbb{R}^2$. Observe that $$z(t)=\left\|\left(t^{\frac{1}{p}} u,(1-t)^{\frac{1}{p}} v\right)\right\|_{\ell_{p}}.$$ For any $x, y \in\Omega\subset{\mathbb{G}}$, we first put $$a=\left(t^{1 / p} u(y),(1-t)^{1 / p} v(y)\right)$$ and $$b=\left(t^{1 / p} u(x),(1-t)^{1 / p} v(y)\right)$$ in the following triangle inequality $$|\|a\|_{\ell_p}-\|b\|_{\ell_p}|\leq\|a-b\|_{\ell_p}$$ and then  divide it by the fractional $p$-kernel $|y^{-1}x|^{Q+ps}$ on both sides followed by  integration to obtain the desired inequality \eqref{ev-z}.
	
We  now proceed to prove the main claim of this theorem. Suppose $v \in X_0^{s, p}(\Omega)$ and $v>0$ in $\Omega$ is a weak solution of \eqref{l-eigen}. Further, by normalizing, if necessary, we may assume that $\|v\|_p=1$. Suppose that $u \in X_0^{s, p}(\Omega)$ minimizes the problem \eqref{ev-1}. In other words 
\begin{equation*}
\lambda_1(\Omega)=\min \left\{I(u): u \in X_0^{s, p}(\Omega), \int_{\Omega}|u(x)|^{p} dx=1\right\}
\end{equation*} is minimized at $u.$
Define, $u_{\epsilon}=u+\epsilon$, $v_{\epsilon}=v+\epsilon$ and
for all $x \in \Omega$ 
\begin{equation*}
z(t,\epsilon)(x)=\left(t u_{\epsilon}(x)^{p}+(1-t) v_{\epsilon}(x)^{p}\right)^{\frac{1}{p}}, ~t \in[0,1].
\end{equation*}
Thanks to the inequality \eqref{ev-z},  the image of $t \mapsto z(t,\epsilon)$ is a family of curves in $X_0^{s, p}(\Omega)$ along which the  energy $I$ is convex. Thus we have

\begin{align}\label{ev-est1}
\iint_{\mathbb{G}\times\mathbb{G}} &\frac{\left|z(t,\epsilon)(x)-z(t,\epsilon)(y)\right|^{p}}{|y^{-1}x|^{Q+ps}} dxdy-\iint_{\mathbb{G}\times\mathbb{G}} \frac{|v(x)-v(y)|^{p}}{|y^{-1}x|^{Q+ps}} dxdy\nonumber\\
& \leq t\left(\iint_{\mathbb{G}\times \mathbb{G}} \frac{|u(x)-u(y)|^{p}}{|y^{-1}x|^{Q+ps}} dxdy -\iint_{\mathbb{G} \times \mathbb{G}} \frac{|v(x)-v(y)|^{p}}{|y^{-1}x|^{Q+ps}} dxdy\right)\nonumber\\
&=t\left(\lambda_1(\Omega)-\nu\right),~\forall\,t\in[0,1] ~\text{and}~ \forall\, \epsilon\ll 1.
\end{align}

Now, using the convexity of $\tau \mapsto|\tau|^{p},$ that is, $(|a|^p-|b|^p \geq p|b|^{p-2} b (a-b))$, we estimate the left hand side  of \eqref{ev-est1} from below as follows:

\begin{align}\label{ev-est2}
\iint_{\mathbb{G}\times\mathbb{G}} &\frac{\left|z(t,\epsilon)(x)-z(t,\epsilon)(y)\right|^{p}}{|y^{-1}x|^{Q+ps}} dxdy-\iint_{\mathbb{G} \times \mathbb{G}} \frac{|v(x)-v(y)|^{p}}{|y^{-1}x|^{Q+ps}} dxdy\nonumber\\
&\geq p\iint_{\mathbb{G} \times \mathbb{G}} \frac{|v(x)-v(y)|^{p-2}(v(y)-v(x))}{|y^{-1}x|^{Q+ps}}\nonumber\\ &\quad \quad\times\left(z(t,\epsilon)(y)-z(t,\epsilon)(x)-(v(y)-v(x))\right) dxdy,
\end{align}
for all\,$t\in[0,1]$ and  for all $ \epsilon\ll 1.$

Observe that, the fact $u, v \in X_0^{s, p}(\Omega)$ implies that $$z(t,\epsilon)\in X_0^{s, p}(\Omega)~\text{and}~v(y)-v(x)=v_{\epsilon}(y)-v_{\epsilon}(x).$$
Thus, on testing \eqref{ev-soln} with $\phi=(z(t,\epsilon)-v_{\epsilon})$ corresponding to the eigenfunction $v$, we get,  for all $\epsilon\ll 1$,

\begin{align}\label{ev-est3}
\iint_{\mathbb{G} \times \mathbb{G}} \frac{|v(x)-v(y)|^{p-2}(v(y)-v(x))}{|y^{-1}x|^{Q+ps}}\left(z(t,\epsilon)(y)-z(t,\epsilon)(x)-\left(v_{\epsilon}(y)-v_{\epsilon}(x)\right)\right)dxdy\nonumber\\
=\nu \int_{\Omega} v(\tau)^{p-1}\left(z(t,\epsilon)(\tau)-v_\epsilon(\tau)\right)d\tau.
\end{align}
Therefore, from \eqref{ev-est1}, \eqref{ev-est2} and \eqref{ev-est3}, we obtain
\begin{equation}\label{ev-est4}
\nu \int_{\Omega} v(\tau)^{p-1} ({z(t,\epsilon)(\tau)-v_{\epsilon}(\tau)}) d\tau\leq t(\lambda_1(\Omega)-\nu),
\end{equation}
for all\,$t\in[0,1]$ and  for all $ \epsilon\ll 1.$

Now, by the concavity $\tau \mapsto |\tau|^{\frac{1}{p}}$ and by recalling that $z(t,\epsilon)(x)=\left(t u_{\epsilon}(x)^{p}+(1-t) v_{\epsilon}(x)^{p}\right)^{\frac{1}{p}}$ we get the following point-wise boundedness a.e. in $\Omega$
\begin{equation}
v(\tau)^{p-1} ({z(t,\epsilon)(\tau)-v_{\epsilon}(\tau)})\geq t\, v(\tau)^{p-1}\left(u_{\epsilon}(\tau)-v_{\epsilon}(\tau)\right)
\end{equation}
and $$v(\tau)^{p-1}\left(u_{\epsilon}(\tau)-v_{\epsilon}(\tau)\right)\in L^1(\Omega).$$
Therefore, from Fatou's lemma we obtain

\begin{align}
\nu \int_{\Omega}\left(\frac{v(\tau)}{v_{\epsilon}(\tau)}\right)^{p-1}((u_{\epsilon}(\tau))^{p}-(v_{\epsilon}(\tau))^{p})d\tau
& \leq \nu\liminf_{t \rightarrow 0^{+}} \int_{\Omega} v(\tau)^{p-1} \frac{z(t,\epsilon)(\tau)-v_{\epsilon}(\tau)}{t}d\tau\nonumber\\&\leq \lambda_1(\Omega)-\nu
\end{align}
for sufficiently small $\epsilon>0.$

Finally, recalling that $v>0$ and applying the Lebesgue dominated convergence theorem and then passing the limit $\epsilon \rightarrow 0^{+}$, we get
\begin{equation}
0 \leq \lambda_1(\Omega)-\nu.
\end{equation}
Since, $\lambda_1(\Omega)$ is the least eigenvalue and $\lambda_1(\Omega)\geq\nu$, we conclude that $\lambda_1(\Omega)=\nu$. Hence, the proof is complete.	
\end{proof}

\begin{lemma}\label{l-simple}
	Let $0<s<1<p<\infty$ and let $\Omega\subset{\mathbb{G}}$ be a bounded domain. Suppose that $u$ and $v$ are two positive eigenfunctions corresponding to $\lambda_1(\Omega)$. Then $u=cv$ for some $c>0$, that means,  $u$ and $v$ are proportional. This says that the first eigenfunction $\lambda_1(\Omega)$ is simple.
\end{lemma}

\begin{proof}
	
	Let $u,v\in X_0^{s, p}(\Omega)$ be such that $\|u\|_p=\|v\|_p=1$ and $u,v \geq 0$. Recall the inequality \eqref{ev-z} for $t=1/2$. Then, we have
	\begin{equation}\label{simple}
		I\left( \left( \frac{v^p + u^p}{2}\right)^{1/p} \right)\leq\frac{I(v)+I(u)}{2}.
	\end{equation}
Observe that $w=\left( \frac{v^p + u^p}{2}\right)^{1/p}\in\mathcal{S}$. Consider the convex function $$B(l,m)=\left|l^{\frac{1}{p}}-m^{\frac{1}{p}}\right|^p~\text{for all}~l>0,m>0.$$

Recall from \cite[Lemma 13]{LL14} that 
$$B\Big(\frac{l_1+l_2}{2}, \frac{m_1+m_2}{2}\Big)\leq \frac{1}{2}B(l_1,m_1) + \frac{1}{2}B(l_2,m_2)$$
and equality holds only if $l_1m_2=l_2m_1$. Thus, using the fact that $u,v,w\in\mathcal{S}$ and \eqref{simple}, we obtain
\begin{align*}
    \lambda_1(\Omega)&\leq \int_{\mathbb{G}} \int_{ \mathbb{G}} \frac{|w(x)-w(y)|^p}{|y^{-1}x|^{Q+ps}}dxdy\\
    &\leq \frac{1}{2}\int_{\mathbb{G}} \int_{ \mathbb{G}} \frac{|u(x)-u(y)|^p}{|y^{-1}x|^{Q+ps}}dxdy + \frac{1}{2}\int_{\mathbb{G}} \int_{ \mathbb{G}} \frac{|v(x)-v(y)|^p}{|y^{-1}x|^{Q+ps}}dxdy=\lambda_1(\Omega).
\end{align*}
Therefore, the inequality becomes equality and thus we get $$u(x)v(y)=v(x)u(y).$$
This implies
	\begin{equation*}
	\frac{u(y)}{v(y)}=\frac{u(x)}{v(x)}=c, (say).
	\end{equation*}
	Hence, $u=c v$ a.e. in $\Omega$. \end{proof}

\noindent Consider the problem
\begin{align}\label{p-kuusi}
\left(-\Delta_{p,{\mathbb{G}}}\right)^s u=0~\text{in}~\Omega\nonumber\\
u=0~\text{in}~{\mathbb{G}}\setminus\Omega.
\end{align}

We say a function $u\in X_0^{s,p}(\Omega)$ is a weak subsolution (or supersolution) of \eqref{p-kuusi}, if for every nonnegative $\phi \in X_0^{s,p}(\Omega)$, we have
\begin{equation}\label{k-subsup}
\int_{\mathbb{G}}\int_{\mathbb{G}}\frac{|u(x)-u(y)|^{p-2}(u(x)-u(y))(\phi(x)-\phi(y))}{|y^{-1}x|^{Q+ps}}dxdy\leq (or \geq) 0.
\end{equation}

A function $u\in X_0^{s,p}(\Omega)$ is a weak solution of \eqref{p-kuusi}, if it is  a weak subsolution as well as a weak supersolution of \eqref{k-subsup}. In particular, for every $\phi \in X_0^{s,p}(\Omega)$, $u$ satisfies
\begin{equation}\label{k-sol}
\int_{\mathbb{G}}\int_{\mathbb{G}}\frac{|u(x)-u(y)|^{p-2}(u(x)-u(y))(\phi(x)-\phi(y))}{|y^{-1}x|^{Q+ps}}dxdy= 0.
\end{equation}

We define the nonlocal tail of a function $v\in X_0^{s,p}({\Omega})$ in a quasi-ball $B_R(x_0)\subset{\mathbb{G}}$ given by 
\begin{equation}
Tail(v,x_0,R)=\left[R^{sp}\int_{{\mathbb{G}}\setminus{B_R(x_0)}}\frac{|v(x)|^{p-1}}{|x_0^{-1}x|^{Q+ps}}dx\right]^{\frac{1}{p-1}}.
\end{equation}
Clearly, for any $v\in L^r({\mathbb{G}}), r\geq p-1$ and $R>0$, we have $Tail(v,x_0,R)$ is finite, by using the H\"{o}lder inequality. For the definitions of the nonlocal tail in the Euclidean space and the Heisenberg group, we refer \cite{DKP14} and \cite{PP22}, respectively.

We state the following comparison principle for fractional $p$-sub-Laplacian on stratified groups.  We refer to  \cite{CZ18, CCHY18} for the strong maximal principle for the subellipic $p$-Laplacian for families of H\"ormander vector fields and to \cite{RS18, RS20, RY22} for a comparison principle for higher order invariant hypoelliptic operators on graded Lie groups.  
\begin{lemma}  \label{ev-weakcomp}
	Let $\lambda>0$, $0< s<1<p<\infty$ and $u, v\in X_0^{s,p}(\Omega)$. Suppose that  $$(-\Delta_{p,{\mathbb{G}}})^sv\geq(-\Delta_{p, {\mathbb{G}}})^su$$ weakly with $v=u=0$ in ${\mathbb{G}}\setminus\Omega$.
	Then $v\geq u$ in ${\mathbb{G}}.$
\end{lemma}
\begin{proof}
It immediately follows from the proof of Lemma \ref{weakcomp} later on with $\lambda=0$.
\end{proof}
The next aim is to establish a minimum principle for the problem \eqref{p-kuusi}. Prior to that we will prove the following logarithmic estimate which will be used to prove the minimum principle.
\begin{lemma}\label{k-log-lemma}
	Let $0<s<1<p<\infty$ and let $u\in X_0^{s,p}(\Omega)$ be a weak supersolution of \eqref{p-kuusi} such that $u\geq0$ in $B_R:=B_R(x_0)\subset\Omega$. Then for any $B_r:=B_r(x_0)\subset B_{\frac{R}{2}}(x_0)$ and for any $d>0$, the following estimate holds:
	\begin{equation}\label{k-log-ineq}
	\int_{B_r}\int_{B_r}\left|\log\frac{u(x)+d}{u(y)+d}\right|^p\frac{dxdy}{|y^{-1}x|^{Q+ps}}\leq Cr^{Q-ps}\left(d^{1-p}\left(\frac{r}{R}\right)^{sp}[Tail(u_{-},x_0,R)]^{p-1}+1\right),
	\end{equation}
	where $C=C(N,p,s)$, $u_{-}$ is the negative part of $u$.
\end{lemma}

\begin{proof}
We follow the idea from \cite{DKP16} which is proved for the Euclidean case. Let us first prove an inequality similar to Lemma 3.1 of \cite{DKP16}. 
	
	Let $p\geq1$ and $\epsilon\in(0,1]$. Then for any $a,b\in\mathbb{R}$, we have
	\begin{equation}
	|a|\leq (|b|+|a-b|).
	\end{equation}
	 Now,  using this triangle inequality and the convexity of $t^p$, we obtain
	\begin{align}\label{k-ineq}
	|a|^p\leq (|b|+|a-b|)^p&= (1+\epsilon)^p\left[\frac{1}{1+\epsilon}|b|+\frac{\epsilon}{1+\epsilon}\frac{|a-b|}{\epsilon}\right]^p\nonumber\\
	&\leq (1+\epsilon)^{p-1}|b|^p+\left(\frac{1+\epsilon}{\epsilon}\right)^{p-1}|a-b|^p\nonumber\\
	&\leq |b|^p +c_p\epsilon|b|^p+c^p(1+c_p\epsilon)\epsilon^{1-p}|a-b|^p,
	\end{align}
	where $c_p=(p-1)\Gamma(\max\{1,p-2\})$ is obtained by iterating the last term of the following estimate
	$$(1+\epsilon)^{p-1}=1+(p-1)\int_{1}^{1+\epsilon}t^{p-2}dt\leq1+\epsilon(p-1)\max\{1, (1+\epsilon)^{p-2}\}.$$
	
We will now proceed to prove the main estimate of this lemma. Let $d>0$ and $\eta \in C_c^{\infty}(\mathbb{G})$ be such that \begin{equation}
0\leq\eta\leq1,~~~\eta\equiv1~\text{in}~B_r,~~~\eta\equiv0~\text{in}~\mathbb{G}\setminus B_{2r}~~~~\text{and}~|\nabla_H \eta|<Cr^{-1}.
\end{equation}	
	Since $u(x)\geq0$ for all $x\in supp(\eta)$,  $\psi=(u+d)^{1-p}\eta^p$ is a well-defined test function for \eqref{k-sol}. Thus, we get
	
\begin{align}\label{k-3.5}
&\int_{B_{2r}} \int_{B_{2r}}\frac{|u(x)-u(y)|^{p-2}(u(x)-u(y))}{|y^{-1}x|^{Q+ps}}\left[\frac{\eta^{p}(x)}{(u(x)+d)^{p-1}}-\frac{\eta^{p}(y)}{(u(y)+d)^{p-1}}\right]dxdy\nonumber\\
&+2\int_{\mathbb{G}\setminus B_{2r}} \int_{B_{2r}}\frac{|u(x)-u(y)|^{p-2}(u(x)-u(y))}{|y^{-1}x|^{Q+ps}}\frac{\eta^{p}(x)}{(u(x)+d)^{p-1}}dxdy=0.
\end{align}

	We will estimate both the terms individually. Set
	
	\begin{align}
	I_1&=\int_{B_{2r}} \int_{B_{2r}}\frac{|u(x)-u(y)|^{p-2}(u(x)-u(y))}{|y^{-1}x|^{Q+ps}}\left[\frac{\eta^{p}(x)}{(u(x)+d)^{p-1}}-\frac{\eta^{p}(y)}{(u(y)+d)^{p-1}}\right]dxdy\label{k-I1}\\
	I_2&=2\int_{\mathbb{G} \setminus B_{2r}} \int_{B_{2r}}\frac{|u(x)-u(y)|^{p-2}(u(x)-u(y))}{|y^{-1}x|^{Q+ps}}\frac{\eta^{p}(x)}{(u(x)+d)^{p-1}}dxdy.\label{k-I2}
	\end{align}
	
	We will first estimate $I_1$. Let us assume that $u(x)>u(y)$. Observe that $u(y)\geq0$ for all $y\in B_{2r}\subset B_R$ using the support of $\eta$. Then on choosing 
	\begin{align} \label{corona}
	    a=\eta(x), b=\eta(y)~\text{and}~\epsilon=l\frac{u(x)-u(y)}{u(x)+d}\in(0,1)~\text{with}~l\in(0,1)
	\end{align}
	in the inequality \eqref{k-ineq}, it can be estimated that

\begin{align}\label{k-3.6}
	&\frac{|u(x)-u(y)|^{p-2}(u(x)-u(y))}{|y^{-1}x|^{Q+ps}}\left[\frac{\eta^{p}(x)}{(u(x)+d)^{p-1}}-\frac{\eta^{p}(y)}{(u(y)+d)^{p-1}}\right]\nonumber\\
	&
		\leq \frac{(u(x)-u(y))^{p-1}}{(u(x)+d)^{p-1}}\frac{ \eta^{p}(y)}{|y^{-1}x|^{Q+ps}}\left[1+c_pl \frac{u(x)-u(y)}{u(x)+d}-\left(\frac{u(x)+d}{u(y)+d}\right)^{p-1}\right]\nonumber\\
		&+c_pl^{1-p} \frac{|\eta(x)-\eta(y)|^{p}}{|y^{-1}x|^{Q+ps}}\nonumber\\
		&= \left(\frac{u(x)-u(y)}{u(x)+d}\right)^{p} \frac{ \eta^{p}(y)}{|y^{-1}x|^{Q+ps}}\left[\frac{1-\left(\frac{u(y)+d}{u(x)+d}\right)^{1-p}}{1-\frac{u(y)+d}{u(x)+d}}+c_pl\right] +c_pl^{1-p} \frac{|\eta(x)-\eta(y)|^{p}}{|y^{-1}x|^{Q+ps}}\nonumber\\
		&:=J_1 +c_pl^{1-p} \frac{|\eta(x)-\eta(y)|^{p}}{|y^{-1}x|^{Q+ps}}.
\end{align}
We now aim to estimate $J_1$. Consider the following function
\begin{equation*}
	h(t):=\frac{1-t^{1-p}}{1-t}=-\frac{p-1}{1-t} \int_{t}^{1} \tau^{-p} d\tau, \quad \forall t \in(0,1).
\end{equation*}

Since, the function $h_1(t)=\frac{1}{1-t} \int_{t}^{1} \tau^{-p} d\tau$ is decreasing in $t\in(0,1)$, we have $h$ is increasing in $t\in(0,1)$. Thus, we have
\begin{equation*}
	h(t) \leq-(p-1),~\forall\, t\in(0,1).
\end{equation*}
{\bf{Case-1:}} $0<t\leq\frac{1}{2}$.\\
In this case, 
\begin{equation*}
	h(t) \leq-\frac{p-1}{2^{p}} \frac{t^{1-p}}{1-t}.
\end{equation*}
For $t=\frac{u(y)+d}{u(x)+d} \in(0,1 / 2]$, i.e. for $u(y)+d \leq \frac{u(x)+d}{2}$, we get 
\begin{equation}\label{k-3.7}
	J_{1} \leq \left(c_pl-\frac{p-1}{2^{p}}\right)\left[\frac{u(x)-u(y)}{u(y)+d}\right]^{p-1}\frac{ \eta^{p}(y)}{|y^{-1}x|^{Q+ps}},
\end{equation}
since $$(u(x)-u(y))\left(\frac{(u(y)+d)^{p-1}}{(u(x)+d)^{p}} \right)=\left(\frac{u(y)+d}{u(x)+d}\right)^{p-1} - \left(\frac{u(y)+d}{u(x)+d}\right)^{p}\leq 1.$$
On choosing $l$ as
\begin{equation}\label{k-3.8}
	l=\frac{p-1}{2^{p+1} c_p} \left( =\frac{1}{2^{p+1} \Gamma (\text{max} \{1, p-2\})}<1\right),
\end{equation}
we obtain
\begin{equation*}
	J_{1} \leq-\frac{p-1}{2^{p+1}} \left[\frac{u(x)-u(y)}{u(y)+d}\right]^{p-1}\frac{ \eta^{p}(y)}{|y^{-1}x|^{Q+ps}}.
\end{equation*}

{\bf{Case-2:}} $\frac{1}{2}<t<1$.\\
Again choosing, $t=\frac{u(y)+d}{u(x)+d} \in(1 / 2,1)$, i.e. $u(y)+d>\frac{u(x)+d}{2}$, we obtain

\begin{align}\label{k-3.9}
	J_{1} &\leq [c_pl-(p-1)]\left[\frac{u(x)-u(y)}{u(x)+d}\right]^{p} \frac{ \eta^{p}(y)}{|y^{-1}x|^{Q+ps}}\nonumber\\
	&- \frac{\left(2^{p+1}-1\right)(p-1)}{2^{p+1}}\left[\frac{u(x)-u(y)}{u(x)+d}\right]^{p} \frac{ \eta^{p}(y)}{|y^{-1}x|^{Q+ps}}
\end{align}
for the choice of $l$ as in \eqref{k-3.8}.

We note that, for $2(u(y)+d)<u(x)+d$, we have 
\begin{equation}\label{k-3.10}
	\left[\log \left(\frac{u(x)+d}{u(y)+d}\right)\right]^{p} \leq c_p\left[\frac{u(x)-u(y)}{u(y)+d}\right]^{p-1},
\end{equation}
and, for $2(u(y)+d) \geq u(x)+d,$ we derive
\begin{equation}\label{k-3.11}
	\left[\log \left(\frac{u(x)+d}{u(y)+d}\right)\right]^{p}=\left[\log \left(1+\frac{u(x)-u(y)}{u(y)+d}\right)\right]^{p} \leq 2^{p}\left(\frac{u(x)-u(y)}{u(x)+d}\right)^{p},
\end{equation}
by using $u(x)>u(y)$ and $\log (1+x)\leq x, ~ \forall x\geq 0$.

Thus, from the estimates \eqref{k-3.6}, \eqref{k-3.7}, \eqref{k-3.9}, \eqref{k-3.10} and \eqref{k-3.11}, we obtain
	\begin{align*}
		&\frac{|u(x)-u(y)|^{p-2}(u(x)-u(y))}{|y^{-1}x|^{Q+ps}}\left[\frac{\eta^{p}(x)}{(u(x)+d)^{p-1}}-\frac{\eta^{p}(y)}{(u(y)+d)^{p-1}}\right]\\
		& \leq-\frac{1}{c_p}\left[\log \left(\frac{u(x)+d}{u(y)+d}\right)\right]^{p} \frac{ \eta^{p}(y)}{|y^{-1}x|^{Q+ps}}+c_pl^{1-p} \frac{|\eta(x)-\eta(y)|^{p}}{|y^{-1}x|^{Q+ps}}.
	\end{align*}
This is true also for  $u(y)>u(x)$ by exchanging $x$ and $y$. The case $u(x)=u(y)$ holds trivially. Thus, we can estimate $I_1$ in \eqref{k-I1} as
	\begin{align}\label{k-3.12}
		I_{1} \leq &-\frac{1}{c(p)} \int_{B_{2 r}} \int_{B_{2 r}} \left|\log \left(\frac{u(x)+d}{u(y)+d}\right)\right|^{p} \frac{ \eta^{p}(y)}{|y^{-1}x|^{Q+ps}}dxdy\nonumber \\
		&+c(p) \int_{B_{2 r}} \int_{B_{2 r}} \frac{|\eta(x)-\eta(y)|^{p}}{|y^{-1}x|^{Q+ps}} dxdy,
	\end{align}
	for some constant $c(p)$ depending on the choice of $l$.
	
We will now estimate the term $I_{2}$ in \eqref{k-I2}. Observe that $u(y)\geq 0$ for all  $y \in B_{R}$. Thus, using $(u(x)-u(y))_{+}\leq u(x)$, we get
\begin{equation}\label{4.40}
	\frac{(u(x)-u(y))_{+}^{p-1}}{(d+u(x))^{p-1}} \leq 1,~\forall\, x \in B_{2 r}, \,y \in B_{R}.
\end{equation}
On the other hand for $y \in \Omega\setminus B_{R}$, we have
\begin{equation}\label{4.41}
	(u(x)-u(y))_{+}^{p-1} \leq 2^{p-1}\left[u^{p-1}(x)+(u(y))_{-}^{p-1}\right],~\forall\, x\in B_{2r}.
\end{equation}
Then using the inequalities \eqref{4.40} and \eqref{4.41}, we obtain
	\begin{align}\label{k-3.13}
		I_{2} \leq & 2 \int_{B_{R} \setminus B_{2 r}} \int_{B_{2 r}} (u(x)-u(y))_{+}^{p-1}(d+u(x))^{1-p} \frac{ \eta^{p}(x)}{|y^{-1}x|^{Q+ps}}dxdy\nonumber \\
		&+2 \int_{\mathbb{G} \setminus B_{R}} \int_{B_{2 r}} (u(x)-u(y))_{+}^{p-1}(d+u(x))^{1-p} \frac{ \eta^{p}(x)}{|y^{-1}x|^{Q+ps}}dxdy\nonumber \\
		\leq & C(p)\int_{\mathbb{G}\setminus B_{2 r}} \int_{B_{2 r}} \frac{\eta^{p}(x)}{|y^{-1}x|^{Q+ps}}dxdy+C'(p)d^{1-p} \int_{\mathbb{G} \setminus B_{R}} \int_{B_{2 r}} \frac{(u(y))_{-}^{p-1}}{|y^{-1}x|^{Q+ps}} dxdy\nonumber\\
		&\leq C(p) \sup _{x \in B_{2r}} r^{Q} \int_{\mathbb{G}\setminus B_{2 r}} \frac{dy}{|y^{-1}x|^{Q+ps}}+C'(p)d^{1-p}\left|B_{r}\right| \int_{\mathbb{G} \setminus B_{R}} \frac{(u(y))_{-}^{p-1}}{\left|y^{-1}x_{0}\right|^{Q+ps}}dy\nonumber\\
		&\leq C(p) r^{Q-ps} + C'(p)d^{1-p} \frac{r^{Q}}{R^{s p}}\left[\operatorname{Tail}\left(u_{-} ; x_{0}, R\right)\right]^{p-1}\nonumber\\
		&\leq C(p) \int_{B_{2 r}} \int_{B_{2 r}} \frac{|\eta(x)-\eta(y)|^{p}}{|y^{-1}x|^{Q+ps}} dxdy + C(p) r^{Q-ps} + C'(p)d^{1-p} \frac{r^{Q}}{R^{s p}}\left[\operatorname{Tail}\left(u_{-} ; x_{0}, R\right)\right]^{p-1},	
	\end{align}
	for some constants $C(p), C'(p)$ depending on $p$.
	
Therefore, by using \eqref{k-3.12} and \eqref{k-3.13} in \eqref{k-3.5}, we get

	\begin{align}\label{k-3.16}
	\int_{B_{2 r}} \int_{B_{2 r}} &\left|\log \left(\frac{u(x)+d}{u(y)+d}\right)\right|^{p} \frac{\eta^{p}(y)}{|y^{-1}x|^{Q+ps}} dxdy\nonumber \\
	\leq & C \int_{B_{2 r}} \int_{B_{2 r}} \frac{|\eta(x)-\eta(y)|^{p}}{|y^{-1}x|^{Q+ps}} dxdy\nonumber \\
	&+C d^{1-p} r^{Q} R^{-ps}\left[\operatorname{Tail}\left(u_{-} ; x_{0}, R\right)\right]^{p-1}+C r^{Q-ps}.
\end{align}
Again, by using $|\nabla_H\eta|\leq Cr^{-1}$, we have
\begin{align}\label{k-3.17}
	\int_{B_{2 r}} \int_{B_{2 r}} \frac{|\eta(x)-\eta(y)|^{p}}{|y^{-1}x|^{Q+ps}}dxdy & \leq Cr^{-p} \int_{B_{2 r}} \int_{B_{2 r}}|y^{-1}x|^{-Q+p(1-s)}dxdy \leq \frac{C}{p(1-s)} r^{-s p}\left|B_{2 r}\right|.
\end{align}
Therefore, the logarithmic estimate \eqref{k-log-ineq} follows from \eqref{k-3.16} and \eqref{k-3.17}.
\end{proof}

We have now all the ingredients to state the following strong minimum principle.
\begin{theorem}[Strong minimum principle]\label{min p}
Let $0<s<1<p<\infty$ and let $\Omega \subset {\mathbb{G}}$ be an open, connected and bounded subset of a stratified Lie group $\mathbb{G}$. Assume that $u \in X_0^{s, p}(\Omega)$ is a weak supersolution of \eqref{p-kuusi} such that $u \geq 0$ in $\Omega$ and $u \not\equiv 0$ in $\Omega.$ Then $u>0$ a.e. in $\Omega$.	
\end{theorem}
\begin{proof}
	
	Suppose for a moment that $u>0$ a.e. in $K$ for every connected and compact subset of $\Omega$. Since $\Omega$ is connected and $u\not\equiv0$ in $\Omega$, there exists a sequence of compact and connected sets $K_j\subset\Omega$ such that	
	\begin{equation*}
		\left|\Omega \backslash K_{j}\right|<\frac{1}{j}~\text { and }~ u \not \equiv 0 ~\text { in }~ K_{j}.
	\end{equation*}
Thus $u>0$ a.e. in $K_j$ for all $j$. Now passing to the limit as $j\rightarrow\infty$, we get that $u>0$ a.e. $\Omega$. Thus it enough to prove the result stated in the lemma for compact and connected subsets of $\Omega$. Since $K \subset \Omega$ is compact and connected, then there exists $r>0$ such that $K \subset\{x \in \Omega: \operatorname{dist}_{cc}(x, \partial \Omega)>2 r\}$. Again, using the compactness, there exist $x_i\in K$, $i=1,2,...,k,$ such that the quasi-balls $B_{r / 2}\left(x_{1}\right), \ldots B_{r / 2}\left(x_{k}\right)$ cover $K$ and	
	\begin{equation}\label{bf-A4}
		\left|B_{r / 2}\left(x_{i}\right) \cap B_{r / 2}\left(x_{i+1}\right)\right|>0, \quad i=1, \ldots, k-1.
	\end{equation}
	Suppose that $u$ vanishes on a subset of $K$ with positive measure. Then with the help of \eqref{bf-A4}, we conclude that there exists $i \in\{1, \ldots, k-1\}$ such that
	\begin{equation*}
		|Z|:=|\left\{x \in B_{r / 2}\left(x_{i}\right): u(x)=0\right\}|>0.
	\end{equation*}
 For $d>0$ and $x \in B_{r / 2}\left(x_{i}\right)$, define	
	\begin{equation*}
		F_d(x)=\log \left(1+\frac{u(x)}{d}\right).
	\end{equation*}
	Observe that for every $x\in Z$ we have 	
	\begin{equation*}
		F_d(x)=0.
	\end{equation*}
	Thus for every $x \in B_{r/2}\left(x_{i}\right)$ and $y\in Z$ with $x\neq y$ we get	
	\begin{equation*}
		\left|F_d(x)\right|^{p}=\frac{\left|F_d(x)-F_d(y)\right|^{p}}{|y^{-1}x|^{Q+ps}}|y^{-1}x|^{Q+ps} .
	\end{equation*}
	Integrating with respect to $y \in Z$, we get	
		\begin{equation*}
		|Z|\left|F_d(x)\right|^{p} \leq\left(\max _{x, y \in B_{r / 2}\left(x_{i}\right)}|y^{-1}x|^{Q+ps}\right) \int_{B_{r / 2}\left(x_{i}\right)} \frac{\left|F_d(x)-F_d(y)\right|^{p}}{|y^{-1}x|^{Q+ps}} d y.
	\end{equation*}
	 Again integrating with respect to $x \in B_{r / 2}\left(x_{i}\right)$ we deduce the following local Poincar\'{e} inequality:	
		\begin{equation}\label{bf-A5}
		\int_{B_{r/2}\left(x_{i}\right)}\left|F_{d}\right|^{p}dx\leq \frac{r^{Q+ps}}{|Z|} \int_{B_{r/2}\left(x_{i}\right)} \int_{B_{r/2}\left(x_{i}\right)} \frac{\left|F_{d}(x)-F_{d}(y)\right|^{p}}{|y^{-1}x|^{Q+ps}} dx dy.
	\end{equation}
Observe that
\begin{equation*}
	\left|\log \left(\frac{d+u(x)}{d+u(y)}\right)\right|^{p}=\left|F_d(x)-F_d(y)\right|^{p}.
\end{equation*}
Plugging the logarithmic estimate \eqref{k-log-ineq} into the above Poincar\'{e} inequality \eqref{bf-A5} by using the fact that $u_{-}=0$ (hence $Tail(u_{-},x_i,R)=0$), we get 
	\begin{equation}\label{bf-A6}
		\int_{B_{r/2}\left(x_{i}\right)}\left|\log \left(1+\frac{u(x)}{d}\right)\right|^{p} dx \leq C \frac{r^{2Q}}{|Z|}.
	\end{equation}
	Now taking limit $d\rightarrow0$ in \eqref{bf-A6}, we obtain $u=0$ a.e. in $B_{r/2}\left(x_{i}\right).$ Thanks to \eqref{bf-A4}, by repeating this arguments in the quasi-balls $B_{r / 2}\left(x_{i-1}\right)$ and $B_{r / 2}\left(x_{i+1}\right)$ and so on we obtain that $u\equiv0$ a.e. on $K$. This is a contradiction and hence $u>0$ a.e. in $K$. This completes the proof.
\end{proof}

\begin{lemma}\label{ll-t17}
	Let $0<s<1<p<\infty.$ Assume that $\Omega\subset{\mathbb{G}}$ is a bounded domain. Let $u$ be an eigenfunction of \eqref{l-eigen} corresponding to $\nu\neq\lambda_1(\Omega)$. Then we have $\nu(\Omega)>\lambda_1(\Omega_{+})$ and $\nu(\Omega)>\lambda_1(\Omega_{-})$, where $\Omega_{+}=\{u>0\}$ and $\Omega_{-}=\{u<0\}$. In particular,	
	\begin{equation}\label{ll-ineq}
	\nu \geq C(N, p, s)\left|\Omega_{+}\right|^{-\frac{ps}{Q}} \text { and } \nu \geq C(Q, p,s) \left|\Omega_{-}\right|^{-\frac{ps}{Q}}.
	\end{equation}
\end{lemma}

\begin{proof}
	Since $\nu\neq\lambda_1(\Omega)$, then $u$ must be sign-changing. On testing the equation \eqref{ev-soln} with $\phi=u_{+}$ we obtain

	\begin{align*}
		\nu \int_{\Omega_{+}}\left|u_{+}\right|^{p} dx &\geq  \iint_{\mathbb{G} \times \mathbb{G}} \frac{\left|u_{+}(x)-u_{+}(y)\right|^{p}}{|y^{-1}x|^{Q+ps}} dxdy +2^{p/2} \iint_{\mathbb{G}\times\mathbb{G}} \frac{\left(u_{+}(y) u_{-}(x)\right)^{\frac{p}{2}}}{|y^{-1}x|^{Q+ps}} dxdy.
	\end{align*}
	Dividing both sides by $\int_{\Omega_{+}}\left|u_{+}(x)\right|^{p} dx$, we have
	\begin{align*}
		\nu &\geq \lambda_{1}\left(\Omega_{+}\right) + 2^{p / 2} \frac{\iint_{\mathbb{G}\times\mathbb{G}} \frac{\left(u_{+}(y) u_{-}(x)\right)^{\frac{p}{2}}}{|y^{-1}x|^{Q+ps}} d x d y}{\int_{\Omega_{+}}\left|u_{+}(x)\right|^{p} dx}.
	\end{align*}
 Therefore, we get $\nu>\lambda_{1}\left(\Omega_{+}\right)$. Inequality \eqref{dadaineq} yields that 
 \begin{align}
     \nu \int_{\Omega_{+}}\left|u_{+}\right|^{p} dx &\geq  \iint_{\mathbb{G} \times \mathbb{G}} \frac{\left|u_{+}(x)-u_{+}(y)\right|^{p}}{|y^{-1}x|^{Q+ps}} dxdy \geq C |\Omega_+|^{-\frac{ps}{Q}} \int_{\Omega_+} |u_+(x)|^p dx
 \end{align}
 and dividing by $\int_{\Omega_+} |u_+(x)|^p dx$ we deduce $\nu \geq C(N,p,s)\left|\Omega_{+}\right|^{-\frac{ps}{Q}}.$ 
 
 Similarly, we can deduce $\nu>\lambda_{1}\left(\Omega_{-}\right)$ and $\nu \geq C\left|\Omega_{-}\right|^{-\frac{ps}{Q}}$. This completes the proof.	\end{proof}

\begin{lemma}\label{isolated}
	Let $0<s<1<p<\infty.$ Assume that $\Omega\subset{\mathbb{G}}$ is bounded. Then the first eigenvalue $\lambda_1(\Omega)$ of \eqref{l-eigen} is isolated.
\end{lemma}
\begin{proof}
	We will prove it by contradiction. Let $\{\nu_{k}\}$ be a sequence of eigenvalues converging to $\lambda_{1}$ such that $\nu_{k} \neq \lambda_{1}$. Suppose that $u_{k}$ is the eigenfunction corresponding to $\nu_k$. Without loss of generality, we may assume that $\|u_{k}\|_p=1.$ Then we have
	\begin{equation*}
	 \nu_{k}=\int_{\Omega\times\Omega} \frac{\left|u_{k}(x)-u_{k}(y)\right|^{p}}{|y^{-1}x|^{Q+ps}} dxdy.
	\end{equation*}
By Theorem \ref{l-3}, there exists $u \in$ $X_0^{s, p}(\Omega)$ such that, up to a subsequence
	
	\begin{equation*}
	u_{k} \rightarrow u ~\text {strongly in }~ L^{p}\left(\Omega\right)~\text {and}~u_{k}(x)\rightarrow u(x) ~\text {point-wise a.e. in }~\Omega.
	\end{equation*}
	Then by applying Fatou's lemma, we get	
	\begin{equation*}
	\frac{\iint_{\mathbb{G}\times\mathbb{G}} \frac{|u(y)-u(x)|^{p}}{|y^{-1}x|^{Q+ps}} dxdy}{\int_{\Omega}|u(x)|^{p} dx} \leq \lim _{k\rightarrow \infty} \nu_{k}=\lambda_1(\Omega).
	\end{equation*}
	Hence we can conclude that $u$ coincides with the first eigenfunction. Theorem \ref{main-thm} infers that $u$ cannot change sign. Thus either $u>0$ in $\Omega$ or $u<0$ in $\Omega$. Thanks to Theorem \ref{ll-t17}, we conclude that $u_{k}$ must change signs in $\Omega$, since $\nu_{k}>\lambda_1(\Omega)$. Therefore, the sets $\Omega^{k}_{\pm}\neq\emptyset$ are with positive measure, where
	
	\begin{equation*}
	\Omega^{k}_{+}=\left\{x\in\Omega: u_{k}(x)>0\right\} \text { and } \Omega^{k}_{-}=\left\{x\in\Omega: u_{k}(x)<0\right\}.
	\end{equation*}
 From the estimate \eqref{ll-ineq}, we have
	
	\begin{equation*}
	\nu_{k} \geq \lambda_{1}\left(\Omega^{k}_{+}\right) \geq C\left|\Omega^{k}_{+}\right|^{-\frac{ps}{Q}}~\text{and}~
	\nu_{k} \geq \lambda_{1}\left(\Omega^{k}_{-}\right) \geq C\left|\Omega^{k}_{-}\right|^{-\frac{ps}{Q}}.
	\end{equation*}
	This implies that
	\begin{equation*}
	|\Omega_{\pm}|=|\lim \sup \Omega^{k}_{\pm}|>0.
	\end{equation*}
	Therefore, letting $k\rightarrow\infty$, we get that $u \geq 0$ in $\Omega_{+}$ and $u \leq 0$ in $\Omega_{-}$. Thus we arrive at a contradiction that $u$ is a first eigenfunction.
\end{proof}

{{\it{Proof of Theorem \ref{ev-mainthm}:}}} The proof immediately follows from the Theorem \ref{l-simple}, Lemma \ref{sign-changing} and Lemma \ref{isolated}.

%We now list out the properties of the first eigenvalue and eigenfunction studied in this section.
%\begin{theorem}\label{ev-properties}
%	Let $0<s<1<p<\infty$ and $\Omega\subset{\mathbb{G}}$ be bounded domain, then we have the following properties.
%	\begin{enumerate}[label=(\roman*)]
%		\item The first eigenfunction is strictly positive.
%		\item The first eigenvalue $\lambda_1(\Omega)$ is simple and the corresponding eigenfunction $\phi_1$ is the only eigenfunction of constant sign. That is if $u$ is an eigenfunction associated to an eigenvalue $\nu>\lambda_1(\Omega)$, then $u$ must be sign-changing.
%		
%		\item The first eigenvalue $\lambda_1(\Omega)$ is isolated.
%	\end{enumerate}	
%\end{theorem}

\section{Nehari Manifold,  weak formulation and multiplicity result}\label{sec5}

In this section, we use the results established in the previous two sections to study  the existence and multiplicity of weak solutions to the  nonlocal singular subelliptic problem \eqref{problem}. We employ the Nehari manifold technique to establish the multiplicity of solutions. The following subsection is devoted to defining the notion of weak solutions, fibering maps, Nehari manifold and some preliminary results.

\subsection{Weak solution and geometry of Nehari manifold}
Let us now present the notion of a positive weak solution to the problem \eqref{problem}.
\begin{definition}\label{d-weak}
	We say that $u \in X_0^{s,p}(\Omega)$ is a positive weak solution of \eqref{problem} if $u> 0$ on $\Omega$ (i.e. $\operatorname{essinf}_{K} u\geq C_K>0$ for all compact subsets $K\subset \Omega$) and
	\begin{equation}
		\langle\left(-\Delta_{p,{\mathbb{G}}}\right)^s u,\psi\rangle-\lambda\int_{\Omega} f(x)u^{-\delta}\psi dx-\int_{\Omega} g(x)u^{q}\psi d x=0
	\end{equation}
	 for all $\psi\in C_c^{\infty}(\Omega)$
\end{definition}

The energy functional $I_{\lambda}: X_0^{s,p}(\Omega)\rightarrow \mathbb{R}$ associated with the problem \eqref{problem}  is defined as
\begin{equation}\label{d-energy}
	I_{\lambda}(u)=\frac{1}{p} \|u\|_{X_0^{s,p}(\Omega)}^{p}-\frac{\lambda}{1-\delta} \int_{\Omega} f(x)|u|^{1-\delta} dx-\frac{1}{q+1} \int_{\Omega} g(x)|u|^{q+1} d x.
\end{equation}
We note here that due to the presence of the singular exponent $\delta\in(0,1)$, the functional $I_{\lambda}$ is not Fr\'echet differentiable. Also,   it is not bounded from below in $X_0^{s,p}(\Omega)$ as $q>p-1$. The method of Nehari manifold plays an important role to extract critical points of this type of energy functional. We define the Nehari manifold $\mathcal{N_{\lambda}}$ for $\lambda>0$ as
\begin{equation}\label{d-nehari}
	\mathcal{N}_{\lambda}:=\left\{u \in X_0^{s,p}(\Omega)\setminus\{0\}:\left\langle I_{\lambda}^{\prime}(u), u\right\rangle=0\right\}.
\end{equation}
We set
\begin{equation}\label{d-min}
	c_{\lambda}=\inf \left\{I_{\lambda}(u): u \in \mathcal{N}_{\lambda}\right\}.
\end{equation}
It is obvious that $u \in \mathcal{N}_{\lambda}$ if and only if 
\begin{equation}\label{d-nehari1}
	\|u\|_{X_0^{s,p}(\Omega)}^{p}-\lambda \int_{\Omega} f(x)|u|^{1-\delta} dx-\int_{\Omega} g(x)|u|^{q+1} dx=0.
\end{equation}
In the next result we establish the coerciveness and boundedness of the functional $I_\lambda.$ 
\begin{lemma}\label{l-coercive}
	For each $\lambda>0$, the energy $I_{\lambda}$ is coercive and bounded from below on $\mathcal{N}_{\lambda}$.
\end{lemma}
\begin{proof}
	By referring  to  the equations \eqref{d-energy} and \eqref{d-nehari1}, we obtain
	
	\begin{align}
		I_{\lambda}(u) &=\frac{1}{p} \|u\|_{X_0^{s,p}(\Omega)}^{p}-\frac{\lambda}{1-\delta} \int_{\Omega} f(x)|u|^{1-\delta} d x-\frac{1}{q+1} \int_{\Omega} g(x)|u|^{q+1} dx\nonumber \\
		&=\left(\frac{1}{p}-\frac{1}{q+1}\right)\|u\|_{X_0^{s,p}(\Omega)}^{p}-\lambda\left(\frac{1}{1-\delta}-\frac{1}{q+1}\right)\int_{\Omega} f(x)|u|^{1-\delta} d x\nonumber\\
		& \geq \left(\frac{1}{p}-\frac{1}{q+1}\right) \|u\|_{X_0^{s,p}(\Omega)}^{p}-c\lambda\|f\|_{\infty}\left(\frac{1}{1-\delta}-\frac{1}{q+1}\right)\|u\|_{X_0^{s,p}(\Omega)}^{1-\delta}.\label{eq3.6}
	\end{align}
	Since $0<1-\delta<1$ and $q+1>p>1$, we conclude that that $I_{\lambda}$ is coercive and bounded from below on $\mathcal{N}_{\lambda}$.	
\end{proof}

Now, we prove the following lemma proceeding as in the proof given in \cite{HSS08}.
\begin{lemma}\label{l-hirano}
	For every non-negative $u\in X_0^{s,p}(\Omega)$ there exists a non-negative, increasing sequence $\{u_n\}$ in $X_0^{s,p}(\Omega)$ with each $u_n$ having compact support in $\Omega$ such that $u_n\rightarrow u$ strongly in $X_0^{s,p}(\Omega)$.		
\end{lemma}
\begin{proof}
	Take $u\in X_0^{s,p}(\Omega)$ and $u\geq0$. By invoking the  density of $C_c^{\infty}(\Omega)$ in $X_0^{s,p}(\Omega)$, we can choose a sequence $\{v_n\}\subset C_c^{\infty}(\Omega)$ converging strongly to $u$ in $X_0^{s,p}(\Omega)$ such that $v_n\geq0$ for all $n\in\mathbb{N}$. We now construct another sequence $\{w_n\}$ by $w_n=\min\{v_n, u\}$. Then $w_n\rightarrow u$ strongly in $X_0^{s,p}(\Omega)$. Let $\epsilon>0$. Choose $n_1>0$ such that $\|w_{n_1}-u\|<\epsilon$, then $\|\max\{u_1, w_n\}-u\|\rightarrow0$, where $u_1:=w_{n_1}$. Again choose, $n_2$ such that $\|\max\{u_1, w_{n_2}\}-u\|<\frac{\epsilon}{2}$, then for $u_2:=\max\{u_1,w_{n_2}\}$ we have $\|\max\{u_2, w_n\}-u\|\rightarrow0$. Continuing in this way, set $u_k=\max\{u_{k-1},w_{n_k}\}$. Note that each $u_k$ is compactly supported and $\|u_k-u\|\leq\frac{\epsilon}{k}$. Thus we can deduce that $\|u_n-u\|\rightarrow0$ and this is the desired sequence.
\end{proof}

For each $u \in X_0^{s,p}(\Omega)$, the fiber map $\phi_{u}: (0,\infty) \rightarrow \mathbb{R}$ is defined by $\phi_{u}(t)=I_{\lambda}(t u)$. This fibering map is an important tool to extract the critical points of the energy functional $I_{\lambda}$ which was first coined by Dr\'abek and Pohozaev \cite{DP97}. Clearly, for $t>0$, we have 
\begin{align} \label{fiber-1}
	&\phi_{u}(t)= \frac{t^{p}}{p}\|u\|^{p}-\lambda \frac{ t^{1-\delta}}{1-\delta} \int_{\Omega}f(x)|u|^{1-\delta} dx-\frac{t^{q+1}}{q+1} \int_{\Omega}g(x)|u|^{q+1} dx, 
	\end{align}
	\begin{align} \label{fiber-2}
	&\phi_{u}^{\prime}(t)=t^{p-1}\|u\|^{p}-\lambda t^{-\delta} \int_{\Omega}f(x)|u|^{1-\delta} d x-t^{q} \int_{\Omega}g(x)|u|^{q+1} d x, 
	\end{align}
	and 
	\begin{align}\label{fiber-3}
	&\phi_{u}^{\prime \prime}(t)=(p-1) t^{p-2}\|u\|^{p}+\delta \lambda t^{-\delta-1} \int_{\Omega}f(x)|u|^{1-\delta} d x-q t^{q-1} \int_{\Omega}g(x)|u|^{q+1} d x.
\end{align}

Observe that $\phi_{u}^{\prime}(t)=\frac{1}{t}\langle I_{\lambda}'(tu), tu\rangle$. Thus  $\phi_{u}^{\prime}(t)=0$ if and only if $tu\in \mathcal{N}_{\lambda}$ for some $t>0$ and $u$ is a critical point of $I_{\lambda}$ if and only if $\phi_{u}^{\prime}(1)=0$. Thus it is natural to split $\mathcal{N}_{\lambda}$ into three essential subsets corresponding to local minima, local maxima and points of inflexion. 

For this purpose, we define the following three sets
\begin{align} \label{n-plus}
	\mathcal{N}_{\lambda}^{+}&=\left\{u \in \mathcal{N}_{\lambda}: \phi_{u}^{\prime}(1)=0, \phi_{u}^{\prime \prime}(1)>0\right\} \nonumber \\&=\left\{t_{0} u \in \mathcal{N}_{\lambda}: t_{0}>0, \phi_{u}^{\prime}\left(t_{0}\right)=0, \phi_{u}^{\prime \prime}\left(t_{0}\right)>0\right\},  
	\end{align}
	\begin{align} \label{n-minus}
	\mathcal{N}_{\lambda}^{-}&=\left\{u \in \mathcal{N}_{\lambda}: \phi_{u}^{\prime}(1)=0, \phi_{u}^{\prime \prime}(1)<0\right\}\nonumber\\ &=\left\{t_{0} u \in \mathcal{N}_{\lambda}: t_{0}>0, \phi_{u}^{\prime}\left(t_{0}\right)=0, \phi_{u}^{\prime \prime}\left(t_{0}\right)<0\right\}, 
	\end{align}
	and 
	\begin{align} 
	\mathcal{N}_{\lambda}^{0}&=\left\{u \in \mathcal{N}_{\lambda}: \phi_{u}^{\prime}(1)=0, \phi_{u}^{\prime \prime}(1)=0\right\}.\label{n-0}
\end{align}
Therefore, it is enough to find two members  $u\in \mathcal{N}_{\lambda}^{+}\setminus\mathcal{N}_{\lambda}^{0}$ and $v\in \mathcal{N}_{\lambda}^{-}\setminus\mathcal{N}_{\lambda}^{0}$ to establish our result. It is easy to see that  only members of the sets $\mathcal{N}_{\lambda}^{\pm}\setminus\mathcal{N}_{\lambda}^{0}$ are critical points of the energy functional $I_{\lambda}$.

We first introduce the following quantity 
\begin{align}\label{lambda-1}
   \Lambda_{1}=  \sup_{u \in X_0^{s,p}(\Omega)} \Big\{\lambda>0: \phi_{u}(t) ~&\text {(ref. \eqref{fiber-1}) has two critical points in }~(0, \infty) \nonumber\Big\}.
\end{align}

\begin{proposition}\label{lambda-finite}
	Under the assumptions on the problem \eqref{problem}, we have $0<\Lambda_1<\infty$.
\end{proposition}
To prove Proposition \ref{lambda-finite} we first prove the following result which ensure that $\Lambda_1>0.$ We first define the function $m_{u}: \mathbb{R}^{+} \rightarrow \mathbb{R}$ by
	
	\begin{equation}\label{fn-m}
		m_{u}(t)=t^{p-1+\delta} \|u\|_{X_0^{s,p}(\Omega)}^p-t^{q+\delta} \int_{\Omega} g(x)|u|^{q+1} dx.
	\end{equation}
	The function $m_u$ will play a crucial role to find a $\lambda_*>0$ in the following lemma.
\begin{lemma}\label{n-nonempty}
	Under the assumptions on the problem \eqref{problem}, there exists $\lambda_*>0$ such that, for every $0<\lambda<\lambda_*,$ we have $\mathcal{N}_{\lambda}^{\pm}\neq\emptyset$, i.e., there exist unique $t_{1}$ and $t_{2}$ in $(0, \infty)$ with $t_{1}<t_{2}$ such that $t_{1} u \in \mathcal{N}_{\lambda}^{+}$and $t_{2} u \in \mathcal{N}_{\lambda}^{-}$.
	Moreover, for any $\lambda \in\left(0, \Lambda_{1}\right)$, we have $\mathcal{N}_{\lambda}^{0}=\emptyset$.
	
	Furthermore, $\sup\limits_{u \in \mathcal{N}_{\lambda}^{+}}\|u\|_{X_0^{s,p}(\Omega)}<\infty$ and $\inf\limits_{v \in \mathcal{N}_{\lambda}^{-}}\|v\|_{X_0^{s,p}(\Omega)}>0$.
\end{lemma}
\begin{proof}
	
	Using \eqref{fiber-2} and \eqref{fn-m} we first deduce that, for $t>0$, we have
	\begin{equation}\label{rel-m-phi}
		\phi_{u}^{\prime}(t)=t^{-\delta}\left(m_{u}(t)-\lambda \int_{\Omega} f(x)|u|^{1-\delta} d x\right).
	\end{equation}
	
	This implies that $\phi_{u}^{\prime}(t)=0$ if and only if  $m_{u}(t)-\lambda \int_{\Omega} f(x)|u|^{1-\delta} dx=0$. Referring to \eqref{fn-m} and $q>p-1$, we note that for $u\ne0$,  $m_{u}(0)=0$ and $\lim _{t \rightarrow \infty} m_{u}(t)=-\infty$. Thus,  one can verify that the function $m_{u}(t)$ attains its maximum at $t=t_{\max}$ given by 
	
	\begin{equation}\label{t-max}
		t_{\max }=\left[\frac{(p-1+\delta)\|u\|_{X_0^{s,p}(\Omega)}^{p}}{(q+\delta) \int_{\Omega}g(x)|u|^{q+1}dx}\right]^{\frac{1}{q+1-p}}.
	\end{equation}
	
	The value of $m_u$ at $t=t_{\max}$ is given by 
		\begin{equation}\label{m-max}
		m_{u}\left(t_{\max }\right)=\left(\frac{q+2-p}{p-1+\delta}\right)\left(\frac{p-1+\delta}{q+\delta}\right)^{\frac{\delta+q}{q+1-p}}\frac{\|u\|_{X_0^{s,p}(\Omega)}^{\frac{p(q+\delta)}{q+1-p}}}{\left(\int_{\Omega}g(x)|u|^{q+1}dx\right)^{\frac{p-1+\delta}{q+1-p}}}.
	\end{equation}
	
	In addition,  by using the fact that $\lim_{t\rightarrow0^+}m_u'(t)>0$, we  conclude that $m_u$ is increasing function on $(0, t_{\max })$ and is  decreasing function on $(t_{\max}, \infty)$. Indeed,  we have
	
	\begin{align}\label{lambda-choice}
		\frac{m_{u}(t_{\max})}{\int_{\Omega}f(x)|u|^{1-\delta}dx} &=\left(\frac{q+2-p}{p-1+\delta}\right)\left(\frac{p-1+\delta}{q+\delta}\right)^{\frac{\delta+q}{\delta+1-p}}\frac{\|u\|_{X_0^{s,p}(\Omega)}^{\frac{p(q+\delta)}{\delta+1-p}}}{\left(\int_{\Omega}g|u|^{q+1}dx\right)^{\frac{p-1+\delta}{q+1-p}}\left(\int_{\Omega}f|u|^{1-\delta}dx\right)}\nonumber\\
		&\geq \left(\frac{q+2-p}{p-1+\delta}\right)\left(\frac{p-1+\delta}{q+\delta}\right)^{\frac{\delta+q}{\delta+1-p}} \frac{S_{q+1}^{-\frac{p-1+\delta}{q+1-p}}{S_{1-\delta}^{-1}}}{\|f\|_{\infty}\|g\|_{\infty}^{\frac{p-1+\delta}{q+1-p}}}, 
	\end{align}
	where $S_{\alpha}=\sup\{\|u\|_{\alpha}^{\alpha}: u\in X_0^{s,p}(\Omega), \|u\|_{X_0^{s,p}(\Omega)}=1\}$ for $\alpha\geq0$, i.e. $\int_{\Omega}|u|^{\alpha}dx\leq S_{\alpha}\|u\|_{X_0^{s,p}(\Omega)}^{\alpha}.$ 
	
	Now we set 
	\begin{equation}\label{lambda-max}
		\lambda_*= \left(\frac{q+2-p}{p-1+\delta}\right)\left(\frac{p-1+\delta}{q+\delta}\right)^{\frac{\delta+q}{\delta+1-p}}\frac{S_{q+1}^{-\frac{p-1+\delta}{q+1-p}}{S_{1-\delta}^{-1}}}{\|f\|_{\infty}\|g\|_{\infty}^{\frac{p-1+\delta}{q+1-p}}}.
	\end{equation}
	Then,  for every $\lambda\in (0,\lambda_*),$ we have
	\begin{equation}
		0<\lambda \int_{\Omega} f(x)|u|^{1-\delta} d x\leq m_{u}\left(t_{\max }\right).
	\end{equation}
	Thus,  there exist $t_1$ and $t_2$ with $0<t_{1}<t_{\max }<t_{2}$ such that 
	\begin{equation}
		m_{u}(t_1)=m_{u}(t_2)=\lambda \int_{\Omega} f(x)|u|^{1-\delta} d x.
	\end{equation}
	Therefore, we deduce that $\phi_u$ decreasing on the set $(0,t_1)$,  increasing on $(t_1, t_2)$ and again  decreasing on $(t_2, \infty)$. So, $\phi_u$ has a local maxima at $t=t_2$ and a local minima at $t=t_1$ such that $t_{2} u \in \mathcal{N}_{\lambda}^{-}$ and $t_{1} u \in \mathcal{N}_{\lambda}^{+}$. In particular, we have 
	\begin{equation}
		I_{\lambda}(t_1)(u)=\min_{0\leq t\leq t_{\max}}I_{\lambda}(u)~~~\text{and}~~~I_{\lambda}(t_2)(u)=\max_{t\geq0}I_{\lambda}(u).
	\end{equation}
	
	We now intend to prove that $\mathcal{N}_{\lambda}^{0}=\emptyset.$ For a moment, we suppose that $u \not \equiv 0$ and $u \in \mathcal{N}_{\lambda}^{0}$, then $u \in \mathcal{N}_{\lambda}$. Therefore, by using the definition of the fibering map  $\phi_{u}(t)$, we see that $t=1$ is a critical point. Now, the above arguments imply that the critical points of $\phi_{u}$ are corresponding to a local minima or a local maxima. Thus, we get either $u \in \mathcal{N}_{\lambda}^{+}$ or $u \in \mathcal{N}_{\lambda}^{-}$. This contradicts the fact that $u \in \mathcal{N}_{\lambda}^{0}$ and therefore we conclude that $\mathcal{N}_{\lambda}^{0}=\emptyset.$
	
	Finally, we assume that $u \in \mathcal{N}_{\lambda}^{+}$. From \eqref{fiber-3} and $\phi_{u}^{''}(1)>0$ we get
	$$(q+1-p)\|u\|_{X_0^{s,p}(\Omega)}^{p}\leq \lambda(q+\delta) c_1\|f\|_{\infty}\|u\|_{X_0^{s,p}(\Omega)}^{1-\delta},$$ which implies that
	\begin{align}\label{l-est1}
		\|u\|_{X_0^{s,p}(\Omega)} &\leq\left(\frac{\lambda(q+\delta) c_1\|f\|_{\infty}}{q+1-p}\right)^{\frac{1}{p-1+\delta}}.
	\end{align}
	
	Similarly, for $v \in \mathcal{N}_{\lambda}^{-}$, from \eqref{fiber-3} and the fact $\phi_{v}^{''}(1)<0$ we obtain
	$$(p-1+\delta)\|v\|_{X_0^{s,p}(\Omega)}^{p}\geq (q+\delta) c_2\|g\|_{\infty}\|v\|_{X_0^{s,p}(\Omega)}^{q+1}$$ which eventually gives 
	\begin{align}\label{l-est2}
		\|v\|_{X_0^{s,p}(\Omega)} &\geq\left(\frac{p-1+\delta}{(q+\delta) c_1\|g\|_{\infty}}\right)^{\frac{1}{q+1-p}}.
	\end{align}
	From \eqref{l-est1} and \eqref{l-est2}, we conclude that $\sup\limits_{u \in \mathcal{N}_{\lambda}^{+}}\|u\|_{X_0^{s,p}(\Omega)}<\infty$ and $\inf\limits_{v \in \mathcal{N}_{\lambda}^{-}}\|v\|_{X_0^{s,p}(\Omega)}>0$. 
\end{proof}

\begin{lemma}\label{l-crit}
	Let u be a local minimizer for $I_{\lambda}$ on $\mathcal{N}_{\lambda}^{-}$ or $\mathcal{N}_{\lambda}^{+}$ such that $u \notin \mathcal{N}_{\lambda}^{0}$. Then $u$ is a critical point of $I_{\lambda}$.
\end{lemma}
\begin{proof}
We first introduce the functional  $J_{\lambda}(u)=\langle I_{\lambda}'(u), u\rangle$. Then,  one can easily verify that $\mathcal{N}_{\lambda}=J_{\lambda}^{-1}(0)\setminus\{0\}$ and 
	
	\begin{align*}
		\left\langle J_{\lambda}'(u), u\right\rangle &=p \|u\|_{X_0^{s,p}(\Omega)}^p-\lambda (1-\delta) \int_{\Omega} f(x)|u|_{X_0^{s,p}(\Omega)}^{1-\delta} d x-q \int_{\Omega} g(x)|u|^{q+1} dx \\
		&=(p-1+\delta)\|u\|^{p}-(q-\delta) \int_{\Omega} h(x)|u|^{q+1} dx,~~\forall~u \in \mathcal{N}_{\lambda}.
	\end{align*}
	Since $u$ is a local minimizer for $I_{\lambda}$ on $\mathcal{N}_{\lambda}$ we can redefine the minimization problem under the following constrained equation
	\begin{equation}
		J_{\lambda}(u)=\langle I_{\lambda}'(u), u\rangle=0
	\end{equation}
	Therefore, the method of Lagrange multipliers guarantees the existence of a constant $\kappa \in \mathbb{R}$ such that
	$$J_{\lambda}'(u)=\kappa I_{\lambda}'(u).$$
	Thus, we obtain
	$$
	\langle I_{\lambda}'(u), u\rangle=\kappa\langle J_{\lambda}'(u), u\rangle=\kappa\phi^{''}(1)=0.
	$$
	Therefore, we conclude that $\kappa=0$ as $u \notin \mathcal{N}_{\lambda}^{0}$. Hence, $u$ is a critical point of $I_{\lambda}$.
\end{proof}

\subsection{Existence of minimizers on $\mathcal{N}_{\lambda}^{+}$ and $\mathcal{N}_{\lambda}^{-}$}

In this subsection, we will prove the existence of minimizers $u_{\lambda}$ and $v_{\lambda}$ of $I_{\lambda}$ on $\mathcal{N}_{\lambda}^{+}$ and $\mathcal{N}_{\lambda}^{-}$ which is attained in $\mathcal{N}_{\lambda}^{+}$ and $\mathcal{N}_{\lambda}^{-}$ respectively. Also, we show that these minimizers are  solutions of \eqref{problem} and $u_{\lambda}\neq v_{\lambda}$. We have the following lemma.

\begin{lemma}\label{soln1}
	For all $\lambda \in\left(0, \Lambda_{1}\right)$, there exists $u_{\lambda} \in \mathcal{N}_{\lambda}^{+}$ such that $I_{\lambda}\left(u_{\lambda}\right)=\inf I_{\lambda}\left(\mathcal{N}_{\lambda}^{+}\right)$. Moreover, $u_{\lambda}$ is a non-negative weak solution to the problem \eqref{problem}.
\end{lemma}
\begin{proof} Since the functional $I_{\lambda}$ is bounded below on $\mathcal{N}_{\lambda}$ (hence bounded below on $\mathcal{N}_{\lambda}^{+}$), there exists a sequence $\{u_n\}\subset \mathcal{N}_{\lambda}^{+}$ such that $I_{\lambda}(u_n) \rightarrow \inf I_{\lambda}(\mathcal{N}_{\lambda}^{+})$ as $n\rightarrow +\infty$. Moreover, by using the coercivity of $I_{\lambda}$ and Lemma \ref{n-nonempty}, we have that  $\{u_n\}$ is bounded in $X_0^{s,p}(\Omega)$ and hence by the reflexiveness of $X_0^{s,p}(\Omega)$, there exists $u_{\lambda}\in X_0^{s,p}(\Omega)$ such that $u_n\rightharpoonup u_{\lambda}$ weakly in $X_0^{s,p}(\Omega)$. Thus,  by the compact embedding (ref. Theorem \ref{l-3}), we get $u_n\rightarrow u_{\lambda}$ strongly in $L^r(\Omega)$ for $1\leq r<p_s^*$ and $u_n\rightarrow u_{\lambda}$ pointwise a.e. in $\Omega$. Our  aim is to show $u_n\rightarrow u_{\lambda}$ strongly in $X_0^{s,p}(\Omega)$. Prior to that we prove that $\inf I_{\lambda}(\mathcal{N}_{\lambda}^{+})<0$. Indeed, for $w\in \mathcal{N}_{\lambda}^{+}$, the fiber map $\phi$ has a local minima in $\mathcal{N}_{\lambda}^{+}$ and $\phi^{''}(1)>0$. Thus, from \eqref{fiber-3}, we get 
	\begin{equation}\label{e1}
		\left(\frac{p-1+\delta}{\delta+q}\right)\left\|w\right\|_{X_0^{s,p}(\Omega)}^{p}>\int_{\Omega}\left|w\right|^{q+1} d x .
	\end{equation}
	The above inequality \eqref{e1} with the fact that $q>p-1$ retrieves the required claim. In fact, we have 
	\begin{align*}
		I_{\lambda}\left(w\right) &=\left(\frac{1}{p}-\frac{1}{1-\delta}\right)\left\|w\right\|_{X_0^{s,p}(\Omega)}^{p}+\left(\frac{1}{1-\delta}-\frac{1}{q+1}\right) \int_{\Omega}\left|w\right|^{q+1} dx\\
		&\leq\frac{(1-\delta-p)}{p(1-\delta)}\left\|w\right\|_{X_0^{s,p}(\Omega)}^{p}+\frac{(p-1+\delta)}{(q+1)(1-\delta)}\left\|w\right\|_{X_0^{s,p}(\Omega)}^p\\
		&=\left(-\frac{1}{p}+\frac{1}{q+1}\right)\left(\frac{p-1+\delta}{1-\delta}\right)\left\|w\right\|_{X_0^{s,p}(\Omega)}^{p}\\
		&=\left(\frac{p-(q+1)}{p(q+1)}\right)\left(\frac{p-1+\delta}{1-\delta}\right)\left\|w\right\|_{X_0^{s,p}(\Omega)}^{p}\\
		&<0.
	\end{align*}
	We now prove the strong convergence by contradiction. Suppose the strong convergence $u_n\rightarrow u_{\lambda}$ in $X_0^{s,p}(\Omega)$ fails. Then we have 
	\begin{equation}
		\|u_{\lambda}\|_{X_0^{s,p}(\Omega)}<\lim\inf\limits_{n\rightarrow\infty}\|u_n\|_{X_0^{s,p}(\Omega)}.
	\end{equation} 	
	Further, by the compact embedding (see Theorem \ref{l-3}), we have
	\begin{align}
		\int_{\Omega}g(x)|u_{\lambda}|^{q+1}dx=\lim\inf\limits_{n\rightarrow\infty}\int_{\Omega}g(x)|u_n|^{q+1} dx\label{est1}\\
		\int_{\Omega}f(x)|u_{\lambda}|^{1-\delta}dx=\lim\inf\limits_{n\rightarrow\infty}\int_{\Omega}f(x)|u_n|^{1-\delta} dx\label{est2}.
	\end{align}	
	Since $\{u_n\}\subset \mathcal{N}_{\lambda}^{+}$ then $\phi'(1)=\langle I_{\lambda}'(u_n), u_n\rangle=0.$ Thus, we get from  \eqref{eq3.6} that
	\begin{align}
		I_{\lambda}(u_n)&\geq \left(\frac{1}{p}-\frac{1}{q+1}\right) \|u_n\|_{X_0^{s,p}(\Omega)}^{p}-c\lambda\|f\|_{\infty}\left(\frac{1}{1-\delta}-\frac{1}{q+1}\right)\|u_n\|_{X_0^{s,p}(\Omega)}^{1-\delta}.
	\end{align}
	Therefore, passing to the limit as $n\rightarrow\infty,$ we deduce
	\begin{align}
		\inf I_{\lambda}(\mathcal{N}_{\lambda}^{+})&\geq \lim\limits_{n\rightarrow\infty} \left(\frac{1}{p}-\frac{1}{q+1}\right) \|u_n\|_{X_0^{s,p}(\Omega)}^{p}-\lim\limits_{n\rightarrow\infty}c\lambda\|f\|_{\infty}\left(\frac{1}{1-\delta}-\frac{1}{q+1}\right)\|u_n\|_{X_0^{s,p}(\Omega)}^{1-\delta}\nonumber\\
		&>\left(\frac{1}{p}-\frac{1}{q+1}\right) \|u_{\lambda}\|_{X_0^{s,p}(\Omega)}^{p}-c\lambda\|f\|_{\infty}\left(\frac{1}{1-\delta}-\frac{1}{q+1}\right)\|u_{\lambda}\|_{X_0^{s,p}(\Omega)}^{1-\delta}\nonumber\\
		&>0,
	\end{align}	
	which is impossible since $\inf I_{\lambda}(\mathcal{N}_{\lambda}^{+})<0$. Thus, $u_n\rightarrow u_{\lambda}$ strongly in $X_0^{s,p}(\Omega)$. Finally, we get $\phi_{u_{\lambda}}^{''}(1)>0$ for all $\lambda\in(0,\Lambda_1)$. Hence, we have $u_{\lambda}\in \mathcal{N}_{\lambda}^{+}$ and $I_{\lambda}(u_{\lambda})=I_{\lambda}(\mathcal{N}_{\lambda}^{+})$. Since, $I_{\lambda}(u_{\lambda})=I_{\lambda}(|u_{\lambda}|)$, we can assume that $u_{\lambda}$ is non-negative. Finally, by the Lemma \ref{l-crit}, we deduce that $u_{\lambda}$ is a critical point of $I_{\lambda}(u_{\lambda})$ and hence a weak solution to the problem \eqref{problem}.	
\end{proof}

The next lemma guarantees the existence of a minimizer in $\mathcal{N}_{\lambda}^{-}$.

\begin{lemma}\label{soln2}
	For all $\lambda \in\left(0, \Lambda_{1}\right)$, there exists $v_{\lambda} \in \mathcal{N}_{\lambda}^{-}$ such that $I_{\lambda}\left(v_{\lambda}\right)=\inf I_{\lambda}\left(\mathcal{N}_{\lambda}^{-}\right)$. Moreover, $v_{\lambda}$ is a non-negative weak solution to the problem \eqref{problem}.
\end{lemma}

\begin{proof}
	Proceeding as in the previous Lemma \ref{soln1}, we can assume that there exists a sequence $\{v_n\}\subset \mathcal{N}_{\lambda}^{-}$ such that $I_{\lambda}(v_n) \rightarrow \inf I_{\lambda}(\mathcal{N}_{\lambda}^{-})$ as $n\rightarrow +\infty$ and there exists $v_{\lambda}\in X_0^{s,p}(\Omega)$ such that $v_n\rightharpoonup v_{\lambda}$ weakly in $X_0^{s,p}(\Omega)$. Therefore, the compact embedding (see Theorem \ref{l-3}) guarantees that $v_n\rightarrow v_{\lambda}$ strongly in $L^r(\Omega)$ for $1\leq r<p_s^*$ and $v_n\rightarrow v_{\lambda}$ pointwise a.e. in $\Omega$. Let us first prove that $\inf I_{\lambda}(\mathcal{N}_{\lambda}^{-})>0$. Suppose 
	$z\in \mathcal{N}_{\lambda}$. Therefore, using \eqref{eq3.6}, we get
	\begin{align}
		I_{\lambda}(z) &\geq \left(\frac{1}{p}-\frac{1}{q+1}\right) \|z\|_{X_0^{s,p}(\Omega)}^{p}-c\lambda\|f\|_{\infty}\left(\frac{1}{1-\delta}-\frac{1}{q+1}\right)\|z\|_{X_0^{s,p}(\Omega)}^{1-\delta}\nonumber\\
		&=\|z\|_{X_0^{s,p}(\Omega)}^{1-\delta}\left(\frac{q+1-p}{p(q+1)}\right)\|z\|_{X_0^{s,p}(\Omega)}^{p-1+\delta}-c\lambda\|f\|_{\infty}\left(\frac{q+\delta}{(1-\delta)(q+1)}\right).\label{eq-n-}
	\end{align}
	Now, for any $\lambda<\frac{(q+1-p)(1-\delta)}{cp\|f\|_{\infty}}$ in \eqref{eq-n-}, we get $I_{\lambda}(z)>0$. Since, $\mathcal{N}_{\lambda}^{+}\cap \mathcal{N}_{\lambda}^{-}=\emptyset$ and $\mathcal{N}_{\lambda}^{+}\cup \mathcal{N}_{\lambda}^{-}=\mathcal{N}_{\lambda}$ (ref. Lemma \ref{n-nonempty}), then we must have $z\in \mathcal{N}_{\lambda}^{-}$. Again, for $z\in \mathcal{N}_{\lambda}^{-}$, there exists $t>0$ such that 
	$\phi_z'(tz)=I_{\lambda}'(tz)<0$, since $1-\delta< 1<p<q+1$. This implies $tz\in \mathcal{N}_{\lambda}^{-}$. This is also true  for $v_{\lambda}$. We are now in a state to prove the strong convergence. Suppose the strong convergence $v_n\rightarrow v_{\lambda}$ in $X_0^{s,p}(\Omega)$ fails. Then  proceeding as Lemma \ref{soln1}, we obtain \begin{align}
		I_{\lambda}(tv_{\lambda})&\leq \lim\limits_{n\rightarrow\infty}I_{\lambda}(tv_n)\leq \lim\limits_{n\rightarrow\infty}I_{\lambda}(v_n)=\inf I_{\lambda}\left(\mathcal{N}_{\lambda}^{-}\right).
	\end{align}
	
	This estimate gives the equality $I_{\lambda}(tv_{\lambda})=\inf I_{\lambda}\left(\mathcal{N}_{\lambda}^{-}\right)$, which is a contradiction. Thus, $v_n\rightarrow v_{\lambda}$ strongly in $X_0^{s,p}(\Omega)$ and $I_{\lambda}(v_{\lambda})=I_{\lambda}(\mathcal{N}_{\lambda}^{-})$.	
	Since, $I_{\lambda}(u_{\lambda})=I_{\lambda}(|u_{\lambda}|)$, we can assume that $u_{\lambda}$ is non-negative. Finally, by the Lemma \ref{l-crit}, we deduce that $u_{\lambda}$ is a critical point of $I_{\lambda}(u_{\lambda})$ and hence a weak solution to the problem \eqref{problem}.
\end{proof}
{{\bf{Proof of Proposition \ref{lambda-finite}:}}} Clearly, from Lemma \ref{n-nonempty}, we get $\Lambda_1 >0$. We will prove the boundedness of $\Lambda_1$ by contradiction. Suppose $\Lambda_1= +\infty$. Let $\lambda_1$ be the first eigenvalue of the problem \eqref{p-kuusi} and let $\phi_1$ be the corresponding first eigenfunction. Choose $\bar{\lambda}>0$ such that
\begin{equation}\label{lambda-finite-1}
    \frac{\bar{\lambda}f(x)}{t^{\delta}}+g(x)t^q>(\lambda_1+\epsilon)t^{p-1}
\end{equation}
for all $t \in (0,\infty)$, $x\in\Omega$ and for some $\epsilon \in(0,1).$ Recall the weak solution $u_{\lambda}\in\mathcal{N}_{\lambda}^{+}$. Then for the above choice of $\bar{\lambda}, \bar{u}:=u_{\bar{\lambda}}\in X_0^{s,p}(\Omega)$ is weak supersolution to

\begin{align}\label{lambda-fin-2}
    (-\Delta_{p,{\mathbb{G}}})^su &=(\lambda_1+\epsilon)|u|^{p-2}u~\text{in}~\Omega,\nonumber\\
    u&=0~\text{in}~{\mathbb{G}}\setminus\Omega.
\end{align}
Then we can choose $r>0$ such that $\underline{u}=r\phi_1$ becomes a subsolution to the problem \eqref{lambda-fin-2}. Now by using the boundedness of $\phi_1$, we can choose a smaller $r>0$ (this choice is possible since $r\phi_1$ is a subsolution) such that $\underline{u}\leq\bar{u}$. Now define $w=r\phi_1$ and $w_n\in X_0^{s,p}(\Omega)$ such that $$(-\Delta_{p,{\mathbb{G}}})^sw_k =(\lambda_1+\epsilon)|w_{k-1}|^{p-2}w_{k-1}~\text{in}~\Omega.$$
From Lemma \ref{ev-weakcomp}, for all $x\in\Omega$ we have 
$$r\phi_1=w_0\leq w_1\leq...\leq w_k\leq....\leq u_{\bar{\lambda}}.$$
This shows that $\{w_k\}$ is bounded in $X_0^{s,p}(\Omega)$ and hence from the reflexivity, we conclude that $w_k\rightharpoonup w$ in $X_0^{s,p}(\Omega)$, up to a subsequence. Thus $w$ becomes a weak solution to \eqref{lambda-fin-2}. Since $\lambda_1+\epsilon>\lambda_1$, we arrive at a contradiction to the fact that $\lambda_1$ is simple and isolated. Hence, $\Lambda_1<\infty$.\\

Having developed all the necessary tools now we are ready to prove our main result. \\

 {\it Proof of Theorem \ref{main-thm}:} Set $\Lambda=\min\{\lambda_*, \Lambda_1\}$. Then,  by using the fact $\mathcal{N}_{\lambda}^{+}\cap \mathcal{N}_{\lambda}^{-}=\emptyset$ and $\mathcal{N}_{\lambda}^{+}\cup \mathcal{N}_{\lambda}^{-}=\mathcal{N}_{\lambda}$ together with Lemma \ref{soln1} and Lemma \ref{soln2}, we get two solutions $u_{\lambda}\neq v_{\lambda}$ in $X_0^{s,p}(\Omega)$. In other words, it shows that the problem \eqref{problem} has at least two non-negative solutions for every $\lambda\in(0,\Lambda)$.

\section{Regularity results for the obtained solutions} \label{sec6}

In this section we prove that all nonnegative solutions to the problem \eqref{problem} are uniformly bounded. Let us begin with the following weak comparison principle.
\begin{lemma}[Weak Comparison Principle]\label{weakcomp}
	Let $\lambda>0$, $0<\delta, s<1<p<\infty$ and $u, v\in X_0^{s,p}(\Omega)$. Suppose that  $$(-\Delta_{p,{\mathbb{G}}})^sv-\frac{\lambda f(x)}{v^{\delta}}\geq(-\Delta_{p, {\mathbb{G}}})^su-\frac{\lambda f(x)}{u^{\delta}}$$ weakly with $v=u=0$ in ${\mathbb{G}}\setminus\Omega$.
	Then $v\geq u$ in ${\mathbb{G}}.$
\end{lemma}
\begin{proof}
	It follows from the statement of the lemma that 
	\begin{align}\label{3compprinci}
		\langle(-\Delta_{p,{\mathbb{G}}})^sv,\phi\rangle-\int_{\Omega}\frac{\lambda\phi}{v} dx&\geq\langle(-\Delta_{p,{\mathbb{G}}})^su,\phi\rangle-\int_{\Omega}\frac{\lambda\phi}{u} dx,
	\end{align}
	for all non-negative $\phi\in X_0^{s,p}(\Omega)$. 
	%In particular choose $\phi=(u-v)^{+}$. To this choice, the inequality in \eqref{3compprinci} looks as follows.
	%{\small\begin{align}\label{3compprinci1}
			%		&\langle(-\Delta_{p,{\mathbb{G}}})^sv-(-\Delta_{p,{\mathbb{G}}})^su,(u-v)^{+}\rangle-\int_{\Omega}\lambda(u-v)^{+}\left(\frac{1}{v}-\frac{1}{u}\right) dx\geq 0.
			%\end{align}}
			
			Recall the identity
			\begin{align}\label{comp principle}
				|b|^{p-2}b-|a|^{p-2}a&=(p-1)(b-a)\int_0^1|a+t(b-a)|^{p-2} dt
			\end{align}
			and define, 
			\begin{align}
				&Q(x,y)=\int_0^1|(u(x)-u(y))+t((v(x)-v(y))-(u(x)-u(y)))|^{p-2} dt.
			\end{align}
			Then, by choosing $a=v(x)-v(y)$, $b=u(x)-u(y)$ we have
			\begin{align}
				&|u(x)-u(y)|^{p-2}(u(x)-u(y))\nonumber-|v(x)-v(y)|^{p-2}(u(x)-u(y))\nonumber\\
				&=(p-1)\{(u(y)-v(y))-(u(x)-v(x))\}Q(x,y).
			\end{align}
			Set $\psi=u-v=(u-v)_{+}-(u-v)_{-}$, where $(u-v)_{\pm}=\max\{\pm(u-v),0\}$. Then, for $\phi=(u-v)^{+}$ we  obtain
			\begin{align}\label{3negativity}
				[\psi(x)-\psi(y)][\phi(x)-\phi(y)]&=(\psi^{+}(x)-\psi^{+}(y))^2\geq0.
			\end{align}
			Therefore,  the inequality  \eqref{3negativity} together with the test function $\phi=(u-v_{+})$ yields that
			\begin{align*}
				0&\geq\int_{\Omega}\lambda(u-v)_{+}\left[\frac{1}{v^{\delta}}-\frac{1}{u^{\delta}}\right]\\
				&\geq \langle(-\Delta_{p,{\mathbb{G}}})^su-(-\Delta_{p,{\mathbb{G}}})^sv,(u-v)_{+}\rangle\\
				&=(p-1)\iint_{\mathbb{G}\times \mathbb{G}}\frac{Q(x,y)(\psi^{+}(x)-\psi^{+}(y))^2}{|y^{-1}x|^{Q+ps}}dxdy
				\geq0.
			\end{align*}
			Hence, we have $v\geq u$ a.e. in ${\mathbb{G}}$.
		\end{proof}
		
				\begin{remark}
			It is worth noting that the result of Lemma \ref{weakcomp} also holds for more general nonlocal operator of subelliptic type on  homogeneous Lie groups.			
		\end{remark}

		We recall the following three results from \cite{BP16} which will be useful for establishing subsequent results. 
		\begin{proposition}[\cite{BP16}]\label{beta convex}
			For every $\beta>0$ and $1\leq p<\infty$ we have the following inequality
			$$\left(\frac{1}{\beta}\right)^{\frac{1}{p}}\left(\frac{p+\beta-1}{p}\right)\geq 1.$$
		\end{proposition}
		\begin{proposition}[\cite{BP16}]\label{l infty 1}
			Let $1<p<\infty$ and let $f: \mathbb{R}\rightarrow \mathbb{R}$ to be a $C^{1}$ convex function and $J_{p}(t):=|t|^{p-2}t$. Then, the following inequality
			\begin{equation}\label{bdd est1 remark}
				J_{p}(a-b)\big[AJ_{p}(f'(a))-BJ_{p}(f'(b))\big]\geq(f(a)-f(b))^{p-2}(f(a)-f(b))(A-B),
			\end{equation}
			holds 
			for every $a, b\in \mathbb{R}$ and every $A, B\geq 0.$
		\end{proposition}
		\begin{proposition}[\cite{BP16}]\label{l infty 2}
			Let $1<p<\infty$ and let $h:\mathbb{R}\rightarrow \mathbb{R}$ to be an increasing function. Define
			$$G(t)=\int_{0}^{t}h'(\tau)^{\frac{1}{p}}\d\tau, t\in \mathbb{R}.$$
			Then, we have
			\begin{equation}\label{bdd est2}
				J_{p}(a-b)(h(a)-h(b))\geq|h(a)-h(b)|^{p}.
			\end{equation}
		\end{proposition}
		\noindent The next lemma concludes the boundedness of solutions of the problem \eqref{problem}. We will employ a Moser type iteration to establish our result. 
		\begin{lemma}\label{bounded}
			Suppose $u\in X_0^{s,p}(\Omega)$ is a nonnegative weak solution to the problem \eqref{problem}, then we have $u\in L^{\infty}({\Omega}).$
		\end{lemma}
		\begin{proof}
		Let $\epsilon>0$ be given. Consider the smooth, Lipschitz function $g_{\epsilon}(t)=(\epsilon^2+t^2)^{\frac{1}{2}}$, which is convex and $g_{\epsilon}(t)\rightarrow|t|$ as $\epsilon \rightarrow0$. In addition, we also have $|g'_{\epsilon}(t)|\leq1.$ For each strictly positive $\psi\in C_c^{\infty}(\Omega)$, test the weak formulation \eqref{d-weak} with the test function $\varphi=|g'_{\epsilon}(u)|^{p-2}g'_{\epsilon}(u)\psi$  to obtain the following estimate
		\begin{align}\label{bound est 2.0}
			\langle(-\Delta_{p,{\mathbb{G}}})^sg_{\epsilon}(u),\psi\rangle
			&\leq\int_\Omega\left(\left|\frac{\lambda f(x)}{{u^{\delta}}}+g(x)u^q\right|\right)|g_{\epsilon}'(u)|^{p-1}\psi dx,
		\end{align}
		for all $\psi\in C_c^{\infty}(\Omega)\cap\mathbb{R^+}.$ 	This is immediate from Proposition \ref{l infty 1} by setting $a=u(x), b=u(y), A=\psi(x)$ and $B=\psi(y)$.
		
		Thanks to Fatou's Lemma, by passing to the limit $\epsilon\rightarrow0,$ we deduce
		\begin{align}\label{bound est 2}
			\langle(-\Delta_{p,{\mathbb{G}}})^s(|u|),\psi\rangle&\leq\int_\Omega\left(\left|\frac{\lambda{f(x)}}{{u^{\delta}}}+g(x)u^q\right|\right)\psi dx.
		\end{align}
		The density result guarantees that \eqref{bound est 2} holds also for $\psi\in X_0^{s,p}(\Omega)$. 
		
		For each $k>0$, consider $u_k=\min\{(u-1)^+, k\}\in X_0^{s,p}(\Omega)$. Then, for fixed $\beta>0$ and $\eta>0, $ by testing \eqref{bound est 2} with the test function $\psi=(u_k+\eta)^{\beta}-\eta^{\beta}$  we get
		\begin{align*}
			\iint_{\mathbb{G}\times \mathbb{G}}&\cfrac{||u(x)|-|u(y)||^{p-2}(|u(x)|-|u(y)|)((u_k(x)+\eta)^{\beta}-(u_k(y)+\eta)^{\beta})}{|y^{-1}x|^{Q+ps}}dxdy\\
			&\qquad\leq\int_\Omega\left|\frac{\lambda{f(x)}}{{u^{\delta}}}+g(x)u^q\right|((u_k+\eta)^{\beta}-\eta^{\beta}) dx.
		\end{align*}
	We apply Proposition \ref{l infty 2} with $h(u)=(u_k+\eta)^{\beta}$ to deduce the following estimate:
		\begin{align}\label{bound est 4}
			&\iint_{\mathbb{G}\times \mathbb{G}}\cfrac{|((u_k(x)+\eta)^{\frac{\beta+p-1}{p}}
				-(u_k(y)+\eta)^{\frac{\beta+p-1}{p}})|^{p}}{|y^{-1}x|^{Q+ps}}dxdy\nonumber\\
			&\leq \cfrac{(\beta+p-1)^{p}}{{\beta}p^{p}} \nonumber \\& \quad\times \iint_{\mathbb{G}\times \mathbb{G}}\cfrac{||u(x)|-|u(y)||^{p-2}(|u(x)|-|u(y)|)((u_k(x)+\eta)^{\beta}-(u_k(y)+\eta)^{\beta})}{|y^{-1}x|^{Q+ps}}dxdy\nonumber\\
			&\leq\cfrac{(\beta+p-1)^{p}}{\beta{p}^{p}}\int_{\Omega}\left(\left|\frac{\lambda{f(x)}}{{u^{\delta}}}\right|+|g(x)u^q|\right)\left((u_k+\eta)^{\beta}-\eta^{\beta}\right) dx\nonumber\\	
			&=\cfrac{(\beta+p-1)^{p}}{\beta{p}^{p}} \nonumber \\& \times\left[\int_{\{u\geq1\}}\lambda|f(x)||u|^{-\delta}\left((u_k+\eta)^{\beta}-\eta^{\beta}\right)+\int_{\{u\geq1\}}|g(x)||u|^{q}\left((u_k+\eta)^{\beta}-\eta^{\beta}\right) dx\right]\nonumber\\
			&\leq\cfrac{(\beta+p-1)^{p}}{\beta{p}^{p}}\left[\int_{\{u\geq1\}}\left(\lambda|f(x)|+|g(x)||u|^{q}\right)\left((u_k+\eta)^{\beta}-\eta^{\beta}\right) dx\right]\nonumber\\
			&\leq{2C(\lambda, \|f\|_{\infty},\|g\|_{\infty})}\left(\cfrac{(\beta+p-1)^{p}}{\beta{p}^{p}}\right)\left[\int_{\Omega}|u|^{q}\left((u_k+\eta)^{\beta}-\eta^{\beta}\right) dx\right]\nonumber\\
			&\leq{C'}\left(\cfrac{(\beta+p-1)^{p}}{\beta{p}^{p}}\right){\|u\|^{q}_{p_s^*}}\|(u_k+\eta)^{\beta}\|_{\kappa},
		\end{align}
		where $\kappa=\frac{p_s^*}{p_s^*-q}$. By recalling the fractional Sobolev inequality for fractional $p$-sub-Laplacian  \eqref{embed-cont}, we obtain
		\begin{align}\label{bound est 5}
			&\iint_{\mathbb{G}\times \mathbb{G}}\cfrac{|((u_k(x)+\eta)^{\frac{\beta+p-1}{p}}
				-(u_k(y)+\eta)^{\frac{\beta+p-1}{p}})|^{p}}{|y^{-1}x|^{Q+ps}}dxdy\geq{C}\left\|(u_k+\eta)^{\frac{\beta+p-1}{p}}-\eta^{\frac{\beta+p-1}{p}}\right\|_{p_{s}^*}^{p}
		\end{align}
		for some $C>0$.
		
		 By using triangle inequality with $(u_k+\eta)^{\beta+p-1}\geq\eta^{p-1}(u_k+\eta)^{\beta},$ we have
		\begin{align}\label{bound est 6}
			\left[\int_{\Omega}\left((u_k+\eta)^{\frac{\beta+p-1}{p}}
			-\eta^{\frac{\beta+p-1}{p}}\right)^{p_s^*}dx\right]^{\cfrac{p}{p_s^*}}\geq\left(\frac{\eta}{2}\right)^{p-1}&\left[\int_{\Omega}(u_k+\eta)^{\frac{p_s^*\beta}{p}}\right]^{\cfrac{p}{p_s^*}}-\eta^{\beta+p-1}|\Omega|^{\cfrac{p}{p_s^*}}.
		\end{align}
		Thus, plugging \eqref{bound est 6} into \eqref{bound est 5} and finally from \eqref{bound est 4}, we obtain
		\begin{align}\label{bdd1}
			\left\|(u_k+\eta)^{\frac{\beta}{p}}\right\|^{p}_{p_s^*}
			\leq{C'}\left[C\left(\frac{2}{\eta}\right)^{p-1}\left(\cfrac{(\beta+p-1)^{p}}{\beta{p}^{p}}\right)\|u\|_{p_s^*}^{q}\|(u_k+\eta)^{\beta}\|_{\kappa}+\eta^{\beta}|\Omega|^{\cfrac{p}{p_s^*}}\right].
		\end{align}
		Now, Proposition \ref{beta convex}, estimates \eqref{bound est 4} and \eqref{bdd1}   imply that
		\begin{align}\label{bound est 7}
			\left\|(u_k+\eta)^{\frac{\beta}{p}}\right\|^{p}_{p_s^*}
			&\leq{C'}\left[\frac{1}{\beta}\left(\cfrac{\beta+p-1}{p}\right)^{p}\left\|(u_k+\eta)^{\beta}\right\|_{\kappa}\left(\frac{C\|u\|_{p_s^*}^{q}}{\eta^{p-1}}+|\Omega|^{\cfrac{p}{p_s^*}-\cfrac{1}{\kappa}} \right)\right].
		\end{align}	
		\noindent We are now in a position to employ a Moser type bootstrap argument to establish our claim. For this, choose $\eta>0$ such that $\eta^{p-1}=C\|u\|_{p_s^*}^{r-1}\left(|\Omega|^{\frac{p}{p_s^*}-\frac{1}{\kappa}}\right)^{-1}$. 	We observe that for $\beta\geq1$, we have $\beta^{p}\geq\left(\frac{\beta+p-1}{p}\right)^{p}.$ 
		
		Let us now rewrite the estimate \eqref{bound est 7} by plugging $\chi=\cfrac{p_s^*}{p\kappa}>1$ and $\tau=\beta\kappa$ as follows:
		\begin{align}\label{bound est 8}
			\left\|(u_k+\eta)\right\|_{\chi\tau}\leq\left(C|\Omega|^{\frac{p}{p_s^*}-\frac{1}{\kappa}}\right)^{\frac{\kappa}{\tau}}\left(\frac{\tau}{\kappa}\right)^{\frac{\kappa}{\tau}}\left\|(u_k+\eta)\right\|_{\tau}.
		\end{align}
	We perform $m$ iterations with $\tau_0=\kappa$ and $\tau_{m+1}=\chi\tau_m=\chi^{m+1}\kappa$ on  \eqref{bound est 8} to have
		\begin{align}\label{bound est 9}
			\left\|(u_k+\eta)\right\|_{\tau_{m+1}}&\leq\left(C|\Omega|^{\frac{p}{p_s^*}-\frac{1}{\kappa}}\right)^{\left(\sum\limits_{i=0}^{m}\frac{\kappa}{\tau_i}\right)}\left(\prod\limits_{i=0}^{m}\left(\frac{\tau_i}{\kappa}\right)^{\frac{\kappa}{\tau_i}}\right)^{p-1}\left\|(u_k+\eta)\right\|_{\kappa}\nonumber\\
			&=\left(C|\Omega|^{\frac{p}{p_s^*}-\frac{1}{\kappa}}\right)^{\frac{\chi}{\chi-1}}\left(\chi^{\frac{\chi}{(\chi-1)^2}}\right)^{p-1}\left\|(u_k+\eta)\right\|_{\kappa}.
		\end{align}
		Now, taking the limit as $m\rightarrow\infty$, we obtain
		\begin{equation}\label{bound est 10}
			\left\|u_k\right\|_{\infty}\leq\left(C|\Omega|^{\frac{p}{p_s^*}-\frac{1}{q}}\right)^{\frac{\chi}{\chi-1}}\left(C'\chi^{\frac{\chi}{(\chi-1)^2}}\right)^{p-1}\left\|(u_k+\eta)\right\|_{q}.
		\end{equation}
	Finally, we  use $u_k\leq(u-1)^+$ in \eqref{bound est 10} combined  with the triangle inequality and pass the limit $k\rightarrow\infty$, to obtain
		\begin{equation}\label{bound est 11}
			\left\|(u-1)^+\right\|_{\infty}\leq\left\|u_k\right\|_{\infty}\leq{C}\left(\chi^{\frac{\chi}{(\chi-1)^2}}\right)^{p-1}\left(|\Omega|^{\frac{p}{p_s^*}-\frac{1}{\kappa}}\right)^{\frac{\chi}{\chi-1}}\left(\left\|(u-1)^+\right\|_{\kappa}+\eta|\Omega|^{\frac{1}{\kappa}}\right).
		\end{equation}
		Therefore, we have $u\in L^{\infty}({\Omega})$ and hence the proof.
	\end{proof}
	
	\section{Appendix A: Sobolev-Rellich-Kondrachov type embedding on stratified Lie groups} \label{appA}
	The purpose of this section to prove continuity and compactness of the Sobolev embedding for $X^{s,p}_0(\Omega)$ where $\Omega$ is any open subset of a stratified Lie group $\mathbb{G}.$ We follow the ideas of \cite{NPV12} to establish the continuous embedding whereas the compact embedding will be proved based on the idea originated by \cite{GL92}. Recently, a similar embedding result is obtained for the Rockland operator on graded Lie groups \cite{RTY20}. The embedding results for the fractional Sobolev space $X_0^{s,p}(\Omega)$ over $\mathbb{R}^N$ can be found in \cite {DGV20, FSV15}. We note here that in  \cite{AM18} the authors studied  weighted compact embeddings  for the fractional Sobolev spaces on bounded extension domains of the Heisenberg group   using an approach similar to \cite{NPV12}. Recently, the fractional Sobolev inequality on stratified Lie groups was shown in \cite[Theorem 2]{KD20} (see \cite{KRS20} for fractional logarithmic inequalities on homogeneous Lie groups). Motivated by the above mentioned investigations we prove the continuous and compact embeddings of  $X_0^{s, p}(\Omega)$ into the Lebesgue space $L^r(\Omega)$ for an appropriate range of $r\geq1$. We now  state the embedding result for the space $X_0^{s, p}(\Omega)$ on stratified Lie groups.

\begin{theorem} \label{l-3App} Let $\mathbb{G}$ be a stratified Lie group of homogeneous dimension $Q$, and let $\Omega\subset\mathbb{G}$ be an open set.
Let $0<s<1<p<\infty$ and $Q>sp.$  Then the fractional Sobolev space $X_0^{s, p}(\Omega)$ is continuously embedded in $L^r(\Omega)$ for $p\leq r\leq p_s^*:=\frac{Qp}{Q-sp}$, that is, there exists a positive constant $C=C(Q,s,p, \Omega)$ such that for all $u\in X_0^{s, p}(\Omega)$, we have
$$\|u\|_{L^r(\Omega)}\leq C \|u\|_{X_0^{s,p}(\Omega)}.$$
Moreover, if $\Omega$ is bounded, then the following embedding
\begin{align}
    X_0^{s,p}(\Omega) \hookrightarrow L^r(\Omega)
\end{align}
is continuous for all $r\in[1,p_s^*]$ and is compact for all $r\in[1,p_s^*)$.
\end{theorem} 
\begin{proof}
Let us recall the fractional Sobolev inequality on stratified Lie groups \cite{KD20}, given by
\begin{equation}\label{fullfractionalsobo} 
    \|u\|_{L^{p_s^*}(\mathbb{G})}\leq C \|u\|_{W^{s,p}(\mathbb{G})}.
\end{equation}

Thus, the space $W^{s,p}(\mathbb{G})$ is continuously embedded in $L^{p_s^*}(\mathbb{G})$. Let $r\in(p,p_s^*)$ be such that $\frac{1}{r}=\frac{\theta}{p}+\frac{1-\theta}{p_s^*}$ for some $\theta\in(0,1)$. Then by the interpolation inequality of Lebesgue spaces we have
$$\|u\|_{L^{r}(\mathbb{G})}\leq  \|u\|_{L^{p}(\mathbb{G})}^{\theta}\|u\|_{L^{p_s^*}(\mathbb{G})}^{1-\theta}.$$
Therefore, using Young's inequality with the exponent $\frac{1}{\theta}$ and $\frac{1}{1-\theta}$ we obtain
\begin{align*}
    \|u\|_{L^{r}(\mathbb{G})}\leq&  \|u\|_{L^{p}(\mathbb{G})} + \|u\|_{L^{p_s^*}(\mathbb{G})}\\
    \leq& \|u\|_{L^{p}(\mathbb{G})} + C\|u\|_{W^{s,p}(\mathbb{G})}.
\end{align*}
Thus, we get that the space $W^{s,p}(\mathbb{G})$ is continuously embedded in $L^{r}(\mathbb{G})$ for all $r\in[p,p_s^*]$.

Let $\Omega$ be an open subset of $\mathbb{G}$. Then, for each $u\in X_0^{s,p}(\Omega)$, we have from \eqref{fullfractionalsobo}, as $u=0$ in $\mathbb{G}\setminus \Omega,$ that
\begin{equation}\label{embed-cont}
    \|u\|_{L^{p_s^*}(\Omega)}\leq C \|u\|_{X_0^{s,p}(\Omega)}.
\end{equation}
Thus the space $X_0^{s,p}(\Omega)$ is continuously embedded in $L^{p_s^*}(\Omega)$. Proceeding as above we conclude that the embedding $X_0^{s,p}(\Omega)\hookrightarrow L^{r}(\Omega)$ is continuous for all $r\in[p,p_s^*]$. That is, for all $u\in X_0^{s,p}(\Omega)$ there exists a $C=C(Q,p,s,\Omega)>0$ such that 
\begin{equation}\label{embed}
\|u\|_{L^{r}(\Omega)}\leq C \|u\|_{X_0^{s,p}(\Omega)}~\text{for all}~p\leq r\leq p_s^*.
\end{equation}

In particular, if $\Omega$ is bounded that is $|\Omega|<\infty$, then applying the H\"{o}lder inequality to the inequality \eqref{embed}, we get the continuous embedding for all $r\in[1,p_s^*]$. This concludes the proof of the first part of the theorem.

 Now, we choose $\eta \in C^\infty_c(\mathbb{G})$ such that $\text{supp} \eta \subset \overline{B}_1(0) ,$ $0\leq \eta \leq 1 $ and $\|\eta\|_{L^1(\mathbb{G})}=1.$ 
 For each $\epsilon>0$ and $f \in L^1_{\text{loc}}(\mathbb{G})$, let us define
$$\eta_\epsilon(x)=\frac{1}{\epsilon^Q} \eta(\epsilon^{-1}x)$$
and 
\begin{align}
    T_\epsilon f(x):= f*\eta_\epsilon(x):=\int_{\mathbb{G}} f(x) \eta_\epsilon(x^{-1}y)\, dy.
\end{align} 
Prior to proceeding to show  the compactness of the embedding, we first we prove the following lemma.
\begin{lemma} \label{lemma3.2} Let $\Omega$ be a open bounded subset of $\mathbb{G}.$ Then, for $1\leq r<\infty,$ the set
     $\mathcal{F} \subset L^r(\Omega)$ is relatively compact in $L^r(\Omega)$ if and only if 
     $\mathcal{F}$ is bounded and $\|T_\epsilon f-f\|_{L^r(\Omega)} \rightarrow 0$ uniformly in $f \in \mathcal{F}$ as $\epsilon \rightarrow 0.$
\end{lemma}
\begin{proof}
Suppose that $\mathcal{F}$ is relatively compact in $L^r(\Omega).$ We agree to extend any function $L^r(\Omega)$ to $L^r(\mathbb{G})$ by assigning zero out of $\Omega.$ 
Let $R>0$ and  let $f_1,f_2,\ldots,f_l \in \mathcal{F}$ be such that $\mathcal{F} \subset \cup_{j=1}^l B_R(f_j) \subset L^r(\Omega).$ Then we have 
\begin{align}
    \|f-T_\epsilon f\|_{L^r(\Omega)} \leq  \|f- f_j\|_{L^r(\Omega)}+ \|f_j- T_\epsilon f_j\|_{L^r(\Omega)}+ \|T_\epsilon f_j-T_\epsilon f\|_{L^r(\Omega)}.
\end{align}
Since $T_\epsilon f \rightarrow f$ in $L^r(\Omega)$ as $\epsilon \rightarrow 0$ and $\|T_\epsilon f\|_r \leq \|f\|_{r}$, we have uniform convergence $\|T_\epsilon f-f\|_{L^r(\Omega)} \rightarrow 0$ by passing $\epsilon \rightarrow 0.$

Conversely, we assume that $\mathcal{F}$ is bounded and $\|T_\epsilon f-f\|_{L^r(\Omega)} \rightarrow 0$ uniformly in $f \in \mathcal{F}$ as $\epsilon \rightarrow 0.$ Choose a bounded sequence $(f_n)$  in $\mathcal{F}.$ Thanks to the Banach-Alouglu theorem we can extract a subsequence (again denoted by $(f_n)$) such that $f_n \rightharpoonup f$ weakly in $L^r(\Omega).$ 
We now aim to prove strong convergence. For that we first observe that 
\begin{align} \label{eqq36d}
    \|f_n- f\|_{L^r(\Omega)} \leq  \|f- T_\epsilon f_n\|_{L^r(\Omega)}+ \|T_\epsilon f_n- T_\epsilon f\|_{L^r(\Omega)}+ \|T_\epsilon f-f\|_{L^r(\Omega)}.
\end{align}
It follows from the weak convergence of $f_n \rightharpoonup f$ that, for all $x \in \mathbb{G}$ and for $\epsilon>0,$
we have $\lim_{n \rightarrow \infty} T_\epsilon (f_n-f)(x) \rightarrow 0.$ Again, by H\"older inequality  we have 
\begin{align}
  \|T_\epsilon(f_n-f)\|^r_{L^r(\Omega)}\leq \|\eta_\epsilon\|_{L^1(\mathbb{G})}^r \|f_n-f\|_{L^r(\Omega)}^r    <\infty
\end{align}
and therefore by the Lebesgue dominated convergence theorem we get
\begin{align}
    \int_{\mathbb{G}} |T_\epsilon (f_n-f)(x)|^r dx \rightarrow 0\quad n \rightarrow \infty.
\end{align}
Thus, as $\epsilon \rightarrow 0,$ all three terms on right hand side of \eqref{eqq36d} go to zero with the use of assumption $\|T_\epsilon f-f\|_{L^r(\Omega)} \rightarrow 0$ uniformly in $f \in \mathcal{F}$ as $\epsilon \rightarrow 0.$ Thus, $f_n \rightarrow f$ converges strongly in $L^r(\Omega).$ Hence, $\mathcal{F}$ is relative compact in $L^r(\Omega)$ for all $1\leq r<\infty$. 
\end{proof}

Now, we continue  the proof of Theorem \ref{l-3}.  We emphasise that by assigning $f=0$ in $\mathbb{G} \setminus \Omega$ we have $f \in W^{s, p}(\mathbb{G})$ for every $f \in X_0^{s,p}(\Omega).$  Now, with the help of Lemma \ref{lemma3.2} we prove the relative compactness of a bounded set $\mathcal{F}$ in $X^{s,p}_0(\Omega).$ Recall that $|B_R(x)|=R^Q|B_1(0)|$ (see \cite[p. 140]{FR16}). Therefore, the boundedness of $\mathcal{F}$ in $L^r(\Omega)$ is immediate from the fractional Gagliardo-Nirenberg inequality \cite[Theorem 4.4.1]{RS19},
\begin{align} \label{gag}
    \|f\|_{L^r(\mathbb{G})} \leq C [f]_{s,p}^b \|f\|_{L^q(\mathbb{G})}^{1-b},
\end{align}
where $p>1, q\geq 1, r>0, b \in (0,1]$ satisfy $\frac{1}{r}=b\left(\frac{1}{p}-\frac{s}{Q} \right)+\frac{1-b}{q}.$

Setting,
$$K_\epsilon:=T_\epsilon f-f~\text{for all}~f \in \mathcal{F},$$ we get from the fractional Gagliardo-Nirenberg inequality \eqref{gag}, as $f \in X_0^{s, p}(\Omega)$ and thus $K_\epsilon(x)=0$ for all $x \in \mathbb{G}\backslash \Omega$, that
\begin{align}
    \|K_\epsilon\|_{L^r(\Omega)} \leq  C [K_\epsilon]_{s,p}^b \|K_\epsilon\|_{L^q(\Omega)}^{1-b} ,
\end{align}
where $\frac{1}{r}=b\left(\frac{1}{p}-\frac{s}{Q} \right)+\frac{1-b}{q}.$
Thus, it is sufficient to show that
\begin{align} \label{eq310}
    [K_\epsilon]_{s,p} \leq \|T_\epsilon f-f\|_{X_0^{s,p}(\Omega)} \rightarrow 0.
\end{align}
This means that 
\begin{align} \label{eq311}
   \lim_{\epsilon \rightarrow 0} \int_{\mathbb{G}} \int_{\mathbb{G}} \frac{|(T_\epsilon f-f)(x)-(T_\epsilon f-f)(y)|^p}{|y^{-1}x|^{Q+ps}} dx dy = 0.
\end{align}

Using $\text{supp}(\eta_\epsilon) \subset B_\epsilon(0),$ the H\"older inequality, Tonelli's and Fubini's theorem we obtain
\begin{align} \label{eq312}
    &\int_{\mathbb{G}} \int_{\mathbb{G}} \frac{|(T_\epsilon f-f)(x)-(T_\epsilon f-f)(y)|^p}{|y^{-1}x|^{Q+ps}} dx dy \\&=\int_{\mathbb{G}} \int_{\mathbb{G}} \frac{1}{|y^{-1}x|^{Q+ps}} \Big| \int_{\mathbb{G}} \eta_\epsilon(z)\left(f(z^{-1}x)-f(z^{-1}y) \right) dz-f(x)+f(y) \Big|^p dx dy \nonumber \\&=\int_{\mathbb{G}} \int_{\mathbb{G}} \frac{1}{|y^{-1}x|^{Q+ps}} \Big| \epsilon^{-Q} \int_{B_\epsilon(0)} \eta(\epsilon^{-1} z)\left(f(z^{-1}x)-f(z^{-1}y) \right) dz-f(x)+f(y) \Big|^p dx dy \nonumber \\&=\int_{\mathbb{G}} \int_{\mathbb{G}} \frac{1}{|y^{-1}x|^{Q+ps}} \Big|  \int_{B_1(0)} \eta( z')\left(f((\epsilon z')^{-1}x)-f((\epsilon z')^{-1}y) -f(x)+f(y)\right) dz' \Big|^p dx dy \nonumber \\&\leq |B_1(0)|^{p-1}\int_{\mathbb{G}} \int_{\mathbb{G}}   \left(\int_{B_1(0)} \eta^p(z) \frac{|f((\epsilon z)^{-1}x)-f((\epsilon z)^{-1}y) -f(x)+f(y)|^p}{|y^{-1}x|^{Q+ps}} dz  \right) dx dy \nonumber \\&= |B_1(0)|^{p-1}\int_{B_1(0)}\int_{\mathbb{G} \times \mathbb{G}}   \frac{|f((\epsilon z)^{-1}x)-f((\epsilon z)^{-1}y) -f(x)+f(y)|^p}{|y^{-1}x|^{Q+ps}} \eta^p(z) dx\, dy\, dz\nonumber 
\end{align}
Now, we note that for the Lie group $\mathbb{G} \times \mathbb{G}$ with the Haar measure $dx dy$ using the continuity of translations on $L^p(\mathbb{G} \times \mathbb{G})$ (see \cite[Theorem 20.15]{HR79}) we obtain, for $v \in L^p(\mathbb{G} \times \mathbb{G})$ and $(z, z) \in \mathbb{G} \times \mathbb{G},$ that
\begin{align} \label{eq313}
    \lim_{\epsilon \rightarrow 0}\int_{\mathbb{G} \times \mathbb{G}} |v((\epsilon z, \epsilon z)^{-1}(x, y))-v(x,y)|^p dx dy =0.
\end{align}
Now, fix $z \in B_1(0)$ and set 
$$v(x, y):=\frac{f(x)-f(y)}{|y^{-1}x|^{\frac{Q+ps}{p}}}.$$
Observe that $v \in L^p(\mathbb{G} \times \mathbb{G})$ as $f \in X_0^{s,p}(\Omega).$ Therefore, the property \eqref{eq313} yields
\begin{align}
    \lim_{\epsilon \rightarrow 0}\int_{\mathbb{G} \times \mathbb{G}} \frac{|f((\epsilon z)^{-1}x)-f((\epsilon z)^{-1}y)-f(x)+f(y)|^p}{|y^{-1}x|^{Q+ps}} dx dy = 0.
\end{align}
Thus,
\begin{align}
    \rho_\epsilon(z):=\eta^p(z)\int_{\mathbb{G} \times \mathbb{G}} \frac{|f((\epsilon z)^{-1}x)-f((\epsilon z)^{-1}y)-f(x)+f(y)|^p}{|y^{-1}x|^{Q+ps}} dx dy \rightarrow 0
\end{align} as $\epsilon \rightarrow 0.$
Now for a.e. $z \in B_1(0),$ using the fact that $f \in X_0^{s,p}(\Omega)$ we have 
\begin{align}
    |\rho_\epsilon(z)|\leq 2^{p-1} \eta^p(z) &\Bigg( \int_{\mathbb{G} \times \mathbb{G}} \frac{|f((\epsilon z)^{-1}x)-f((\epsilon z)^{-1}y)|^p}{|y^{-1}x|^{Q+ps}} dx dy \nonumber\\&\quad+\int_{\mathbb{G} \times \mathbb{G}} \frac{|f(x)-f(y)|^p}{|y^{-1}x|^{Q+ps}} dx dy\Bigg)\nonumber \\&=
    2^p \eta^p(z) \int_{\mathbb{G} \times \mathbb{G}} \frac{|f(x)-f(y)|^p}{|y^{-1}x|^{Q+ps}} dx dy.
\end{align}
Observe that the last estimate shows $\rho_\epsilon \in L^\infty(B_1(0))$ uniformly  as $\epsilon \rightarrow 0.$ Therefore, by the Lebesgue dominated convergence theorem we conclude that 
\begin{align}
    \int_{B_1(0)}\int_{\mathbb{G} \times \mathbb{G}}  & \frac{|f((\epsilon z)^{-1}x)-f((\epsilon z)^{-1}y) dz-f(x)+f(y)|^p}{|y^{-1}x|^{Q+ps}} \eta^p(z) dx\, dy\, dz \nonumber \\&=\int_{B_1(0)} \rho_\epsilon(z) dz \rightarrow 0
\end{align} as $\epsilon \rightarrow 0.$
This fact along with \eqref{eq312} gives \eqref{eq311} and so \eqref{eq310}.
Finally, by Lemma \ref{lemma3.2} we conclude that $\mathfrak{F}$ is relatively compact in $L^r(\Omega)$. Thus we conclude that the space $X_0^{s,p}(\Omega)$ is compactly embedded in $L^r(\Omega)$ for all $r\in[1,p_s^*)$.
\end{proof}

\section{Appendix B}
	In this section we prove the following important lemma.
	\begin{lemma}\label{lemma ercole}
	Let $u_1, u_2 \in X_{0}^{s, p}(\Omega) \setminus\{0\}$. Then there exists a positive constant $C=C_p$, depending only on $p$, such that
\begin{equation} \label{append}
		\langle\left(-\Delta_{p,{\mathbb{G}}}\right)^s u_1- \left(-\Delta_{p,{\mathbb{G}}}\right)^s u_2, u_1-u_2\rangle\geq C_p\begin{cases}
			[u_1-u_2]_{s,p}^p,&\text{if}~p\geq2\\
			\frac{[u_1-u_2]_{s,p}^2}{\left([u_1]_{s,p}^p+[u_2]_{s,p}^p\right)^{\frac{2-p}{p}}},&\text{if}~1<p<2.
		\end{cases}
	\end{equation}
	\end{lemma}
\begin{proof}
    Let us recall the well-known Simmon's inequality

\begin{equation}\label{simmon}
\left(|a|^{p-2}a-|b|^{p-2}b\right) \cdot(a-b) \geq C(p) \begin{cases}\frac{|a-b|^{2}}{(|a|+|b|)^{2-p}} & \text { if } \quad 1<p<2 \\ |a-b|^{p} & \text { if } \quad p \geq 2,\end{cases}
\end{equation}

where $a, b \in \mathbb{R}^{N} \setminus\{0\}$ and $C(p)$ is a positive constant depending only on $p$.

 For simplicity we denote
\begin{equation*}
w_i(x, y)=u_i(x)-u_i(y), \quad i=1,2.
\end{equation*}

Therefore,
\begin{equation*}
\langle\left(-\Delta_{p,{\mathbb{G}}}\right)^s u_1- \left(-\Delta_{p,{\mathbb{G}}}\right)^s u_2, u_1-u_2\rangle=\iint_{\mathbb{G}\times\mathbb{G}} \frac{|w_1|^{p-2}w_1-|w_2|^{p-2}w_2}{|y^{-1}x|^{Q+ps}}\left(w_1-w_2\right)dxdy.
\end{equation*}
Observe that for $p\geq 2$ the inequality \eqref{append} immediately follows from the inequality \eqref{simmon}. Thus we are left to establish the inequality \eqref{append} for the range $1<p<2$.

From \eqref{simmon}, we have

\begin{equation}\label{simmin est}
\left\langle\left(-\Delta_{p, \mathbb{G}}\right)^{s} u_1-\left(-\Delta_{p, \mathbb{G}}\right)^{s} u_2, u_1-u_2\right\rangle \geq C(p) \iint_{\mathbb{G}\times\mathbb{G}} \frac{|w_1-w_2|^{2}}{\left(|w_1|+|w_2|\right)^{2-p}|y^{-1}x|^{Q+ps}}dxdy.
\end{equation}

Now from the H\"older's inequality, we get

\begin{align}
[u_1-u_2]_{s,p}^{p} &=\iint_{\mathbb{G}\times\mathbb{G}} \frac{|w_1-w_2|^{p}}{|y^{-1}x|^{Q+ps}}dxdy\nonumber\\
&=\iint_{\mathbb{G}\times\mathbb{G}} \frac{|w_1-w_2|^p}{\left(|w_1|+|w_2|\right)^{\frac{p(2-p)}{2}} |y^{-1}x|^{(Q+ps)\frac{p}{2}}} \frac{\left(|w_1|+|w_2|\right)^{\frac{p(2-p)}{2}}}{|y^{-1}x|^{(Q+ps)\frac{2-p}{2}}}dxdy\leq A^{\frac{p}{2}} B^{\frac{2-p}{2}},
\end{align}
where
\begin{equation*}
    A=\iint_{\mathbb{G}\times\mathbb{G}} \frac{|w_1-w_2|^{2}}{\left(|w_1|+|w_2|\right)^{2-p}|y^{-1}x|^{Q+ps}}dxdy
\end{equation*}
and
\begin{equation*}
    B=\iint_{\mathbb{G}\times\mathbb{G}} \frac{\left(|w_1|+|w_2|\right)^p}{|y^{-1}x|^{Q+ps}}dxdy \leq 2^p\iint_{\mathbb{G}\times\mathbb{G}} \frac{|w_1|^p+|w_2|^p}{|y^{-1}x|^{Q+ps}}dxdy =2^p([u_1]_{s,p}^p +[u_2]_{s,p}^p).
\end{equation*}

From \eqref{simmin est}, we deduce

\begin{align}
\left\langle\left(-\Delta_{p, \mathbb{G}}\right)^{s} u_1-\left(-\Delta_{p, \mathbb{G}}\right)^{s} u_2, u_1-u_2\right\rangle &\geq C(p)A \nonumber \\
& \geq C(p)\left([u_1-u_2]_{s,p}^p B^{-\frac{2-p}{2}}\right)^{\frac{2}{p}} \nonumber\\
& \geq C(p)[u_1-u_2]_{s,p}^2\left(2^p\left([u_1]_{s,p}^p+[u_2]_{s,p}^p\right)\right)^{-\frac{2-p}{p}}\nonumber \\
&=2^{p-2} C(p) \frac{[u_1-u_2]_{s,p}^2}{\left([u_2]_{s,p}^p+[u_2]_{s,p}^p\right)^{\frac{2-p}{p}}},
\end{align}
 completing the proof.
\end{proof}

\section{Conflict of interest statement}
On behalf of all authors, the corresponding author states that there is no conflict of interest.

\section{Data availability statement}
Data sharing not applicable to this article as no datasets were generated or analysed during the current study.

\section*{Acknowledgement}
 The authors are grateful to the reviewer for reading the manuscript carefully, providing several useful comments and suggesting relevant references.     SG would like to thank the Ghent Analysis \& PDE centre, Ghent University, Belgium for the support during his research visit. VK and MR are supported by the FWO Odysseus 1 grant G.0H94.18N: Analysis and Partial
Differential Equations, the Methusalem programme of the Ghent University Special Research Fund (BOF) (Grant number 01M01021) and by FWO Senior Research Grant G011522N. MR is also supported by EPSRC grants
EP/R003025/2 and EP/V005529/1.

%\section*{Disclosure Statement}
%The authors confirm that there are no potential conflict of interest.

%	\section*{References}
%\setlength{\bibsep}{1pt}
%\bibliographystyle{abbrv}
%{\bibliography{Ref_SG_VK_MR}}

\end{document}